\documentclass[10pt,reqno,a4paper,oneside,11pt]{amsart}%
\usepackage{amsfonts}
\usepackage{amsfonts}
\usepackage{amsfonts}
\usepackage{amsfonts}
\usepackage{amsfonts}
\usepackage{amsfonts}
\usepackage{amsfonts}
\usepackage{amsfonts}
\usepackage{amsfonts}
\usepackage{amsfonts}
\usepackage{amsfonts}
\usepackage{amsfonts}
\usepackage{amsfonts}
\usepackage{amsfonts}
\usepackage{amsfonts}
\usepackage{amsfonts}
\usepackage{amsfonts}
\usepackage{amsfonts}
\usepackage{amsfonts}
\usepackage{amsfonts}
\usepackage{mathrsfs}
\usepackage{mathrsfs}
\usepackage{amsfonts}
\usepackage{amssymb}
\usepackage{amsmath}
\usepackage{amsthm}
\usepackage{graphicx}%
\providecommand{\U}[1]{\protect\rule{.1in}{.1in}}
\newtheorem{theorem}{Theorem}[section]
\newtheorem{corollary}[theorem]{Corollary}
\newtheorem{lemma}[theorem]{Lemma}

\newtheorem{example}{Example}

\theoremstyle{definition}
\theoremstyle{remark}
\numberwithin{equation}{section}

\ifx\pdfoutput\relax\let\pdfoutput=\undefined\fi
\newcount\msipdfoutput
\ifx\pdfoutput\undefined\else
\ifcase\pdfoutput\else
\ifx\paperwidth\undefined\else
\ifdim\paperheight=0pt\relax\else\pdfpageheight\paperheight\fi
\ifdim\paperwidth=0pt\relax\else\pdfpagewidth\paperwidth\fi
\fi\fi\fi
\begin{document}
\pagestyle{myheadings}

\begin{center}
{\Large \textbf{Systems of four coupled one sided Sylvester-type real quaternion matrix equations and their applications}}\footnote{This research was supported by
the grants from the National Natural
Science Foundation of China (11571220).
\par
{}* Corresponding author. \par  Email address: hzh19871126@126.com (Z.H. He); wqw@t.shu.edu.cn, wqw369@yahoo.com (Q.W. Wang)}

\bigskip

{ \textbf{Zhuo-Heng He$^{a,b}$, Qing-Wen Wang$^{b,*}$}}

{\small
\vspace{0.25cm}

$a.$ Department of Mathematics and Statistics, Auburn University, AL 36849-5310, USA\\

$b.$ Department of Mathematics, Shanghai University, Shanghai 200444, P. R. China }

\end{center}

\begin{quotation}
\noindent\textbf{Abstract:} In this paper, we derive  some necessary and
sufficient solvability conditions for some systems of one sided coupled Sylvester-type real quaternion matrix equations in terms of
ranks and generalized inverses of matrices.
We also give the expressions of the general solutions to these systems when they are
solvable. Moreover, we provide some
numerical examples to illustrate our results. The findings of this
paper extend some known results in the literature.
\end{quotation}

\section{\textbf{Introduction}}

Quaternions were introduced by Irish mathematician Sir William Rowan Hamilton
in 1843. It is well known that quaternion algebra is an associative and noncommutative division algebra over the real number field.
 Quaternions and quaternion matrices have
found a huge amount of applications in quantum physics, signal and color image processing, and so on (e.g. \cite{N.
LE Bihan}, \cite{S. De Leo}, \cite{Took1}-\cite{Took4}). General properties of quaternions and quaternion matrices can be found in \cite{F. Zhang}. Quaternion matrix equations play
an important role in mathematics and other disciplines, such as engineering, system and control theory. There have been many papers using various approaches to investigate many quaternion matrix equations (e.g. \cite{wanghe2}-\cite{wanghe3}, \cite{QWWangandyushaowen}-\cite{Q.W5}, \cite{yuan4}, \cite{yuan1}, \cite{wangronghao}).

The Sylvester-type matrix equations have wide applications in neural network \cite{zhangyong}, robust control (\cite{mao8}, \cite{Varga}),   output feedback control (\cite{VLS}, \cite{mao10}), the almost noninteracting control by measurement feedback problem (\cite{J.W}, \cite{H.K.Wimmer}), graph theory \cite{Dmytryshyn}, and so on. Since Roth \cite{Roth} first studied the one-sided generalized Sylvester matrix equation
\begin{align*}
AX-YB=C
\end{align*}over the complex field in 1952, there have been many papers to discuss the generalized Sylvester matrix equations (e.g. \cite{jkb}, \cite{JKB2}, \cite{helaa}, \cite{heac2017}, \cite{mao16}, \cite{mao1}, \cite{SangGuLee2}, \cite{mao5}, \cite{wangheauto}-\cite{Q.W12}, \cite{H.K.Wimmer}). For instance, De Ter$\acute{a}$n et al. (\cite{xibanya01}, \cite{xibanyalama}) considered the $\star$-Sylvester equation $AX+X^{*}B=0$ and $AX+BX^{*}=0$. Quite recently, Dmytryshyn and K{\aa}gstr\"{a}m \cite{Dmytryshyn} presented some solvability conditions of the following systems consisting of
Sylvester and $\star$-Sylvester equations through the corresponding equivalence relations of the block matrices
\begin{align*}
  \left\{\begin{array}{c}
A_{i}X_{k}\pm X_{j}B_{i}=C_{i},~i=1,\ldots,n_{1},\\
F_{i^{'}}X_{k^{'}}\pm X_{j^{'}}^{*}G_{i^{'}}=H_{i^{'}},~i^{'}=1,\ldots,n_{2},
\end{array}
  \right.
\end{align*}
where $k,j,k^{'},j^{'}\in\{1,\ldots,m\},$ each unknown $X_{l}$ is $r_{l}\times c_{l},l=1,\ldots,m,$ and all other matrices are of appropriate sizes. Jonsson and K{\aa}gstr\"{a}m (\cite{Jonsson01}, \cite{Jonsson02}) provided some effective approaches for solving one-sided and two-sided triangular Sylvester-type matrix equations.

The study on the coupled generalized Sylvester matrix equations
is active in recent years.  Lee and Vu \cite{SangGuLee} derived a consistency condition for the following system of  mixed Sylvester matrix equations through the corresponding equivalence relations of the block matrices
\begin{align}\label{bufensys02}
A_{1}X_{1}-X_{2}B_{1}=C_{1},\quad A_{2}X_{3}-X_{2}B_{2}=C_{2},
\end{align}where $A_{i},B_{i}$ and $C_{i}(i=1,2)$ are given matrices over a field,  $X_{1},X_{2}$ and $X_{3}$ are unknowns.
Wang and He \cite{wangheauto} gave some computable  necessary and
sufficient solvability conditions for the system (\ref{bufensys02}), and presented the general solution   when (\ref{bufensys02}) is
solvable. Afterwards, He and Wang \cite{shangdaxuebao} provided some necessary and sufficient solvability conditions for the system
of  mixed Sylvester matrix equations
\begin{align}\label{bufensys01}
A_{1}X_{1}-X_{2}B_{1}=C_{1},\quad A_{2}X_{2}-X_{3}B_{2}=C_{2},
\end{align}where $A_{i},B_{i}$ and $C_{i}(i=1,2)$ are given complex matrices,  $X_{1},X_{2}$ and $X_{3}$ are unknowns.
They also derived an expression of the general solution to the system (\ref{bufensys01}). Recently, Wang and He \cite{auto001} considered the following three
systems of generalized coupled Sylvester matrix equations with four variables
\begin{equation}\label{specialsys01}
  \left\{\begin{array}{c}
A_{1}X-YB_{1}=C_{1},\\
A_{2}Z-YB_{2}=C_{2},\\
A_{3}Z-WB_{3}=C_{3},
\end{array}
  \right.
  \end{equation}
\begin{equation}\label{specialsys02}
  \left\{\begin{array}{c}
A_{1}X-YB_{1}=C_{1},\\
A_{2}Y-ZB_{2}=C_{2},\\A_{3}Z-WB_{3}=C_{3},
\end{array}
  \right.
  \end{equation}
\begin{equation}\label{specialsys03}
  \left\{\begin{array}{c}
A_{1}X-YB_{1}=C_{1},\\
A_{2}Y-ZB_{2}=C_{2},\\A_{3}W-ZB_{3}=C_{3},
\end{array}
  \right.
\end{equation}
where $A_{i},B_{i}$ and $C_{i}(i=1,2,3)$ are given complex matrices,  $X,Y,Z$ and $W$ are unknowns. He, Mauricio, Wang and De Moor \cite{helaa} considered two sided coupled generalized Sylvester matrix equations with four variables
\begin{align*}
A_{i}X_{i}B_{i}+C_{i}X_{i+1}D_{i}=E_{i},~i=1,2,3,
\end{align*}where $A_{i},B_{i},C_{i},D_{i},E_{i}, (i=1,2,3)$ are given complex matrices,  $X_{i}$ are unknowns. Very recently, He and Wang \cite{heac2017} derived the solvability conditions and the general solution to the system of the periodic discrete-time coupled Sylvester quaternion matrix equations
\begin{equation*}
  \left\{\begin{array}{c}
A_{k}X_{k}+Y_{k}B_{k}=M_{k},\\
C_{k}X_{k+1}+Y_{k}D_{k}=N_{k},
\end{array}
  \right.(k=1,2),
\end{equation*}
where $A_{k},B_{k},C_{k},D_{k},M_{k},N_{k}$ are given matrices, $X_{k}$ and $Y_{k}$ are unknowns.

To our best knowledge, there has been little information on the solvability and
the general solutions to the systems of four coupled one sided Sylvester-type real quaternion matrix equations with five unknowns.
Motivated by the wide applications of generalized Sylvester matrix equations and real quaternion matrix equations and in order to improve the theoretical
development of generalized Sylvester real quaternion matrix equations, we in this
paper consider  the solvability and the expressions of the general solutions to the following systems of four coupled one sided Sylvester-type real quaternion matrix equations
\begin{align}\label{sys01}
  \left\{\begin{array}{c}
A_{1}X_{1}-X_{2}B_{1}=C_{1},\\
A_{2}X_{3}-X_{2}B_{2}=C_{2},\\
A_{3}X_{3}-X_{4}B_{3}=C_{3},\\
A_{4}X_{4}-X_{5}B_{4}=C_{4},
\end{array}
  \right.
\end{align}
\begin{align}\label{sys02}
  \left\{\begin{array}{c}
A_{1}X_{1}-X_{2}B_{1}=C_{1},\\
A_{2}X_{2}-X_{3}B_{2}=C_{2},\\
A_{3}X_{3}-X_{4}B_{3}=C_{3},\\
A_{4}X_{5}-X_{4}B_{4}=C_{4},
\end{array}
  \right.
\end{align}
\begin{align}\label{sys03}
  \left\{\begin{array}{c}
A_{1}X_{1}-X_{2}B_{1}=C_{1},\\
A_{2}X_{2}-X_{3}B_{2}=C_{2},\\
A_{3}X_{4}-X_{3}B_{3}=C_{3},\\
A_{4}X_{5}-X_{4}B_{4}=C_{4},
\end{array}
  \right.
\end{align}
\begin{align}\label{sys04}
  \left\{\begin{array}{c}
A_{1}X_{1}-X_{2}B_{1}=C_{1},\\
A_{2}X_{2}-X_{3}B_{2}=C_{2},\\
A_{3}X_{3}-X_{4}B_{3}=C_{3},\\
A_{4}X_{4}-X_{5}B_{4}=C_{4},
\end{array}
  \right.
\end{align}
\begin{align}\label{sys05}
  \left\{\begin{array}{c}
A_{1}X_{1}-X_{2}B_{1}=C_{1},\\
A_{2}X_{3}-X_{2}B_{2}=C_{2},\\
A_{3}X_{4}-X_{3}B_{3}=C_{3},\\
A_{4}X_{4}-X_{5}B_{4}=C_{4},
\end{array}
  \right.
\end{align}
where $A_{i},B_{i},C_{i}, (i=1,2,3,4)$ are given real quaternion matrices, and $X_{1},\ldots,X_{5}$ are unknowns. Note that the $i$th equation and $(i+1)$th equation in (\ref{sys01})-(\ref{sys05}) have a common variable $X_{i+1}$. The given real quaternion matrices $A_{i}$ located at the left of the variables and $B_{i}$ located at the right of the variables. Systems (\ref{bufensys02})-(\ref{specialsys03}) are special cases of systems (\ref{sys01})-(\ref{sys05}).

The remainder of the paper is organized as follows. In Section 2, we provide some
known lemmas which are used in this paper. In Section 3,4,5,6,7, we present some solvability
conditions to the systems of four coupled one sided Sylvester-type real quaternion matrix equations (\ref{sys01})-(\ref{sys05}), respectively. We also derive  the general solutions to the systems (\ref{sys01})-(\ref{sys05}), respectively.  Moreover, we give some
numerical examples to illustrate our results.

Throughout this paper, let $\mathbb{R}$   be the real   number fields.
Let $\mathbb{H}^{m\times
n}$ be the set of all $m\times n$ matrices over the real quaternion algebra
\[
\mathbb{H}=\big\{a_{0}+a_{1}\mathbf{i}+a_{2}\mathbf{j}+a_{3}\mathbf{k}\big|~\mathbf{i}^{2}=\mathbf{j}^{2}=\mathbf{k}^{2}%
=\mathbf{ijk}=-1,a_{0},a_{1},a_{2},a_{3}\in\mathbb{R}\big\}.
\]
For $A\in\mathbb{H}^{m\times n}$, the symbols $A^{\ast}$ and $r(A)$ denote the
conjugate transpose and the rank of $A$, respectively. The identity matrix
with appropriate size is denoted by $I$. The Moore-Penrose inverse of
$A\in\mathbb{H}^{m\times n}$, denoted by $A^{\dag}$, is defined to be the
unique solution $X$ to the following four matrix equations
\[
(1)\text{ }AXA=A,\quad(2)\text{ }XAX=X,\quad(3)\text{ }(AX)^{\ast}%
=AX,\quad(4)\text{ }(XA)^{\ast}=XA.
\]
Furthermore, $L_{A}$ and $R_{A}$ stand for the two projectors $L_{A}%
=I-A^{\dag}A$ and $R_{A}=I-AA^{\dag}$ induced by $A$, respectively. It is
known that $L_{A}=L_{A}^{\ast}$ and $R_{A}=R_{A}^{\ast}$.

\section{\textbf{Preliminaries}}

In this section, we review some lemmas which are used in the further development
of this paper. The following lemma give the solvability conditions and general solution to the mixed  Sylvester real quaternion matrix equations (\ref{bufensys02}).

\begin{lemma}\label{lemma01} \cite{wangheauto}
Let $A_{i},B_{i},$ and $C_{i}(i=1,2)$ be given. Set
\begin{align*}
D_{1}=R_{B_{1}}B_{2},A=R_{A_{2}}A_{1},B=B_{2}L_{D_{1}},
C=R_{A_{2}}(R_{A_{1}}C_{1}B_{1}^{\dag}B_{2}-C_{2})L_{D_{1}}.
\end{align*}
Then the following statements are equivalent:\\
$(1)$ The mixed  Sylvester real quaternion matrix equations (\ref{bufensys02}) is consistent.\\
$(2)$
$
R_{A_{1}}C_{1}L_{B_{1}}=0,~R_{A}C=0,~CL_{B}=0.
$\\
$(3)$
\begin{align*}
r\begin{pmatrix}C_{1}&A_{1}\\B_{1}&0\end{pmatrix}=r(A_{1})+r(B_{1}),
\end{align*}
\begin{align*}
r\begin{pmatrix}C_{2}&A_{2}\\B_{2}&0\end{pmatrix}=r(A_{2})+r(B_{2}),
\end{align*}
\begin{align*}
r\begin{pmatrix}B_{2}&B_{1}&0&0\\C_{2}&C_{1}&A_{1}&A_{2}\end{pmatrix}=r(A_{1},~A_{2})
+r(B_{1},~B_{2}).
\end{align*}

In this case, the general solution to the mixed  Sylvester real quaternion matrix equations (\ref{sys01}) can be expressed as
\begin{align*}
X=A_{1}^{\dag}C_{1}+U_{1}B_{1}+L_{A_{1}}W_{1},
\end{align*}
\begin{align*}
Y=-R_{A_{1}}C_{1}B_{1}^{\dag}+A_{1}U_{1}+V_{1}R_{B_{1}},
\end{align*}
\begin{align*}
Z=A_{2}^{\dag}(C_{2}-R_{A_{1}}C_{1}B_{1}^{\dag}B_{2}+A_{1}U_{1}B_{2})+W_{4}D_{1}+L_{A_{2}}W_{6},
\end{align*}
where
\begin{align*}
U_{1}=A^{\dag}CB^{\dag}+L_{A}W_{2}+W_{3}R_{B},
\end{align*}
\begin{align*}
V_{1}=-R_{A_{2}}(C_{2}-R_{A_{1}}C_{1}B_{1}^{\dag}B_{2}+A_{1}U_{1}B_{2})D_{1}^{\dag}+A_{2}W_{4}+W_{5}R_{D_{1}},
\end{align*}
and $W_{1}, \cdots, W_{6}$ are arbitrary matrices over $\mathbb{H}$ with appropriate sizes.
\end{lemma}

The solvability conditions and general solution to the mixed  Sylvester real quaternion matrix equations (\ref{bufensys01}) can be found in the
following lemma.

\begin{lemma}\cite{shangdaxuebao}\label{lemma02}
Let $A_{i},B_{i},$ and $C_{i}(i=1,2)$ be given. Set
\begin{align*}
A_{11}=R_{(A_{2}A_{1})}A_{2},~B_{11}=R_{B_{1}}L_{B_{2}},~C_{11}=R_{(A_{2}A_{1})}(A_{2}R_{A_{1}}C_{1}B_{1}^{\dag}+C_{2})L_{B_{2}}.
\end{align*}
Then the following statements are equivalent:\\
$(1)$ The mixed   generalized Sylvester real quaternion matrix equations (\ref{bufensys01}) is consistent.\\
$(2)$
\begin{align*}
R_{A_{1}}C_{1}L_{B_{1}}=0,~R_{A_{11}}C_{11}=0,~C_{11}L_{B_{11}}=0.
\end{align*}
$(3)$
\begin{align*}
r\begin{pmatrix}C_{i}&A_{i}\\B_{i}&0\end{pmatrix}=r(A_{i})+r(B_{i}),~
r\begin{pmatrix}A_{2}A_{1}&A_{2}C_{1}+C_{2}B_{1}\\0&B_{2}B_{1}\end{pmatrix}=r(A_{2}A_{1})
+r(B_{2}B_{1}).
\end{align*}

In this case, the general solution to the mixed generalized   Sylvester real quaternion matrix equations (\ref{sys01}) can be expressed as
\begin{align*}
X_{1}=A_{1}^{\dag}C_{1}+U_{1}B_{1}+L_{A_{1}}W_{1},
\end{align*}
\begin{align*}
X_{2}=-R_{A_{1}}C_{1}B_{1}^{\dag}+A_{1}U_{1}+V_{1}R_{B_{1}},
\end{align*}
\begin{align*}
X_{3}=-R_{(A_{2}A_{1})}(C_{2}+A_{2}R_{A_{1}}C_{1}B_{1}^{\dag}-A_{2}V_{1}R_{B_{1}})B_{2}^{\dag}+A_{2}A_{1}W_{4}+W_{5}R_{B_{2}},
\end{align*}
where
\begin{align*}
V_{1}=A_{11}^{\dag}C_{11}B_{11}^{\dag}+L_{A_{11}}W_{2}+W_{3}R_{B_{11}},
\end{align*}
\begin{align*}
U_{1}=(A_{2}A_{1})^{\dag}(C_{2}+A_{2}R_{A_{1}}C_{1}B_{1}^{\dag}-A_{2}V_{1}R_{B_{1}})+W_{4}B_{2}+L_{(A_{2}A_{1})}W_{6},
\end{align*}
and $W_{1},\cdots,W_{6}$ are arbitrary matrices over $\mathbb{H}$ with appropriate sizes.
\end{lemma}

Based on Lemma \ref{lemma01}, we can solve the following mixed  Sylvester real quaternion matrix equations
\begin{align}\label{bufensys03}
A_{1}X_{1}-X_{2}B_{1}=C_{1},\quad A_{2}X_{1}-X_{3}B_{2}=C_{2}.
\end{align}

\begin{lemma}\label{lemma05}
Let $A_{i},B_{i},$ and $C_{i}(i=1,2)$ be given. Set
\begin{align*}
A_{11}=R_{(A_{2}L_{A_{1}})}A_{2},~B_{11}=B_{1}L_{B_{2}},~C_{11}=R_{(A_{2}L_{A_{1}})}(C_{2}-A_{2}A_{1}^{\dag}C_{1})L_{B_{2}},
\end{align*}
Then the following statements are equivalent:\\
$(1)$ The mixed  Sylvester real quaternion matrix equations (\ref{bufensys03}) is consistent.\\
$(2)$
$
R_{A_{1}}C_{1}L_{B_{1}}=0,~R_{A_{11}}C_{11}=0,~C_{11}L_{B_{11}}=0.
$\\
$(3)$
\begin{align*}
r\begin{pmatrix}C_{1}&A_{1}\\B_{1}&0\end{pmatrix}=r(A_{1})+r(B_{1}),
\end{align*}
\begin{align*}
r\begin{pmatrix}C_{2}&A_{2}\\B_{2}&0\end{pmatrix}=r(A_{2})+r(B_{2}),
\end{align*}
\begin{align*}
r\begin{pmatrix}C_{1}&A_{1}\\C_{2}&A_{2}\\B_{1}&0\\B_{2}&0\end{pmatrix}=r\begin{pmatrix}A_{1}\\A_{2}\end{pmatrix}+r\begin{pmatrix}B_{1}\\B_{2}\end{pmatrix}.
\end{align*}

In this case, the general solution to the mixed  Sylvester real quaternion matrix equations (\ref{bufensys03}) can be expressed as
\begin{align*}
X_{1}=A_{1}^{\dag}C_{1}+U_{1}B_{1}+L_{A_{1}}U_{2},
\end{align*}
\begin{align*}
X_{2}=-R_{A_{1}}C_{1}B_{1}^{\dag}+A_{1}U_{1}+W_{6}R_{B_{1}},
\end{align*}
\begin{align*}
X_{3}=-R_{(A_{2}L_{A_{1}})}(C_{2}-A_{2}A_{1}^{\dag}C_{1}-A_{2}U_{1}B_{1})B_{2}^{\dag}+A_{2}L_{A_{1}}W_{1}+W_{3}R_{B_{2}},
\end{align*}
where
\begin{align*}
U_{1}=A_{11}^{\dag}C_{11}B_{11}^{\dag}+L_{A_{11}}W_{4}+W_{5}R_{B_{11}},
\end{align*}
\begin{align*}
U_{2}=(A_{2}L_{A_{1}})^{\dag}(C_{2}-A_{2}A_{1}^{\dag}C_{1}-A_{2}U_{1}B_{1})+W_{1}B_{2}+L_{(A_{2}L_{A_{1}})}W_{2},
\end{align*}
and $W_{1}, \cdots, W_{6}$ are arbitrary matrices over $\mathbb{H}$ with appropriate sizes.
\end{lemma}

The following real quaternion matrix equation
\begin{align}\label{equ28}
A_{1}X_{1}+X_{2}B_{1}+C_{3}X_{3}D_{3}+C_{4}X_{4}D_{4}=E_{1}
\end{align}
which play an important role in the construction of the solvability conditions and the general solution to the systems (\ref{sys01})-(\ref{sys05}).

\begin{lemma}\label{lemma03}\cite{hewang4}, \cite{wanghe4444444}
Let $A_{1},B_{1},C_{3},D_{3},C_{4},D_{4}$, and $E_{1}$ be given. Set
\begin{align*}
A  &  =R_{A_{1}}C_{3},B=D_{3}L_{B_{1}},C=R_{A_{1}}C_{4},D=D_{4}L_{B_{1}},\\
E  &  =R_{A_{1}}E_{1}L_{B_{1}},M=R_{A}C,N=DL_{B},S=CL_{M}.
\end{align*}
Then the equation (\ref{equ28}) is consistent if and only if
$$ R_{M}R_{A}E=0,  EL_{B}L_{N}=0, R_{A}EL_{D}=0,  R_{C}EL_{B}=0. $$

In this case, the general solution can be
expressed as\newline%
\begin{align*}
 X_{1}= A_{1}^{\dag}(E_{1}-C_{3}X_{3}D_{3}-C_{4}X_{4}D_{4})-A_{1}^{\dag}%
T_{7}B_{1}+L_{A_{1}}T_{6},
\end{align*}
\begin{align*}
 X_{2}= R_{A_{1}}(E_{1}-C_{3}X_{3}D_{3}-C_{4}X_{4}D_{4})B_{1}^{\dag}%
+A_{1}A_{1}^{\dag}T_{7}+T_{8}R_{B_{1}},
\end{align*}
\begin{align*}
X_{3}=
A^{\dag}EB^{\dag}-A^{\dag}CM^{\dag}EB^{\dag}-A^{\dag}SC^{\dag
}EN^{\dag}DB^{\dag}-A^{\dag}ST_{2}R_{N}DB^{\dag}+L_{A}T_{4}+T_{5}%
R_{B},
\end{align*}
\begin{align*}
 X_{4}= M^{\dag}ED^{\dag}+S^{\dag}SC^{\dag}EN^{\dag}+L_{M}L_{S}T_{1}+L_{M}T_{2}R_{N}+T_{3}R_{D},
\end{align*}
where $T_{1},\ldots,T_{8}$ are arbitrary matrices over $\mathbb{H}$
with appropriate sizes.
\end{lemma}

The following lemma can be easily generalized to $\mathbb{H}$.

\begin{lemma}\label{lemma04}\cite{GPH}
 Let $A\in\mathbb{H}^{m\times n},B\in
\mathbb{H}^{m\times k},C\in\mathbb{H}^{l\times
n},D\in\mathbb{H}^{m\times p},E\in\mathbb{H}^{q\times
n},Q\in\mathbb{H}^{m_{1}\times k}$, and $P\in\mathbb{H}^{l\times
n_{1}}$ be given. Then\newline
$(1)~ r(A)+r(R_{A}B)=r(B)+r(R_{B}A)=r(A,~B).$\newline
$ (2)~ r(A)+r(CL_{A})=r(C)+r(AL_{C})=r%
\begin{pmatrix}
A\\
C
\end{pmatrix}
.$
\end{lemma}

\section{\textbf{Some solvability conditions and the general solution to
system (\ref{sys01})}}
In this section, we consider the solvability conditions and the general solution to the
system of one-sided coupled Sylvester-type real quaternion matrix equations (\ref{sys01}). For simplicity, put
\begin{align}
A_{11}=R_{B_{2}}B_{1},~B_{11}=R_{A_{1}}A_{2},~C_{11}=B_{1}L_{A_{11}},~D_{11}=R_{A_{1}}(R_{A_{2}}C_{2}B_{2}^{\dag}B_{1}-C_{1})L_{A_{11}},
\end{align}
\begin{align}
A_{22}=R_{(A_{4}A_{3})}A_{4},B_{22}=R_{B_{3}}L_{B_{4}},C_{22}=R_{(A_{4}A_{3})}(A_{4}R_{A_{3}}C_{3}B_{3}^{\dag}+C_{4})L_{B_{4}},
\end{align}
\begin{align}
A_{33}=(L_{A_{2}},~-L_{A_{3}}),B_{33}=\begin{pmatrix}R_{C_{11}}B_{2}\\-B_{4}B_{3}\end{pmatrix},A_{44}=-L_{(A_{4}A_{3})},
\end{align}
\begin{align}
E_{1}=A_{3}^{\dag}C_{3}+(A_{4}A_{3})^{\dag}(C_{4}B_{3}+A_{4}R_{A_{3}}C_{3})-A_{2}^{\dag}C_{2}-B_{11}^{\dag}D_{11}C_{11}^{\dag}B_{2},
\end{align}
\begin{align}
A=R_{A_{33}}L_{B_{11}},B=B_{2}L_{B_{33}},C=R_{A_{33}}A_{44},D=B_{3}L_{B_{33}},
\end{align}
\begin{align}
E=R_{A_{33}}E_{1}L_{B_{33}},M=R_{A}C,N=DL_{B},S=CL_{M}.
\end{align}

Now we give the fundamental theorem of this section.

\begin{theorem}\label{theorem01}
Let $A_{i},B_{i},$ and $C_{i}(i=1,2,3,4)$ be given. Then the following statements are equivalent:\\$(1)$ The system of one-sided coupled Sylvester-type real quaternion matrix equations (\ref{sys01}) is consistent.\\$(2)$
\begin{align}\label{equhh037}
r\begin{pmatrix}C_{i}&A_{i}\\B_{i}&0\end{pmatrix}=r(A_{i})+r(B_{i}),\quad (i=1,2,3,4),
\end{align}
\begin{align}\label{equhh038}
r\begin{pmatrix}C_{1}&C_{2}&A_{1}&A_{2}\\B_{1}&B_{2}&0&0\end{pmatrix}=r(A_{1},~A_{2})+r(B_{1},~B_{2}),
\end{align}
\begin{align}\label{equhh039}
r\begin{pmatrix}A_{4}C_{3}+C_{4}B_{3}&A_{4}A_{3}\\B_{4}B_{3}&0\end{pmatrix}=r(A_{4}A_{3})+r(B_{4}B_{3}),
\end{align}
\begin{align}\label{equ034}
r\begin{pmatrix}C_{1}&C_{2}&A_{1}&A_{2}\\0&A_{4}C_{3}+C_{4}B_{3}&0&A_{4}A_{3}\\B_{1}&B_{2}&0&0\\0&B_{4}B_{3}&0&0\end{pmatrix}
=r\begin{pmatrix}A_{1}&A_{2}\\0&A_{4}A_{3}\end{pmatrix}+r\begin{pmatrix}B_{1}&B_{2}\\0&B_{4}B_{3}\end{pmatrix},
\end{align}
\begin{align}\label{equ035}
r\begin{pmatrix}C_{2}&A_{2}\\C_{3}&A_{3}\\B_{2}&0\\B_{3}&0\end{pmatrix}
=r\begin{pmatrix}A_{2}\\A_{3}\end{pmatrix}+r\begin{pmatrix}B_{2}\\B_{3}\end{pmatrix},
\end{align}
\begin{align}\label{equ036}
r\begin{pmatrix}C_{1}&C_{2}&A_{1}&A_{2}\\C_{3}&0&0&A_{3}\\B_{1}&B_{2}&0&0\\0&B_{3}&0&0\end{pmatrix}=
r\begin{pmatrix}A_{1}&A_{2}\\0&A_{3}\end{pmatrix}+r\begin{pmatrix}B_{1}&B_{2}\\0&B_{3}\end{pmatrix},
\end{align}
\begin{align}\label{equ037}
r\begin{pmatrix}C_{2}&A_{2}\\A_{4}C_{3}+C_{4}B_{3}&A_{4}A_{3}\\B_{2}&0\\B_{4}B_{3}&0\end{pmatrix}=
r\begin{pmatrix}A_{2}\\A_{4}A_{3}\end{pmatrix}+r\begin{pmatrix}B_{2}\\B_{4}B_{3}\end{pmatrix}.
\end{align}
$(3)$
\begin{align}\label{equhh314}
R_{A_{2}}C_{2}L_{B_{2}}=0,~D_{11}L_{C_{11}}=0,~R_{B_{11}}D_{11}=0,
\end{align}
\begin{align}\label{equhh315}
R_{A_{3}}C_{3}L_{B_{3}}=0,~R_{A_{22}}C_{22}=0,~C_{22}L_{B_{22}}=0,
\end{align}
\begin{align}\label{equhh316}
R_{M}R_{A}E=0,~EL_{B}L_{N}=0,~R_{A}EL_{D}=0,~R_{C}EL_{B}=0.
\end{align}

In this case, the general solution to the system of one-sided coupled Sylvester-type real quaternion matrix equations (\ref{sys01}) can be expressed as
\begin{align}\label{equhh317}
X_{1}=A_{1}^{\dag}(C_{1}-R_{A_{2}}C_{2}B_{2}^{\dag}B_{1}+A_{2}U_{1}B_{1})+W_{4}A_{11}+L_{A_{1}}W_{6},
\end{align}
\begin{align}
X_{2}=-R_{A_{2}}C_{2}B_{2}^{\dag}+A_{2}U_{1}+V_{1}R_{B_{2}},
~
X_{4}=-R_{A_{3}}C_{3}B_{3}^{\dag}+A_{3}U_{2}+V_{2}R_{B_{3}},
\end{align}
\begin{align}
X_{5}=-R_{(A_{4}A_{3})}(C_{4}+A_{4}R_{A_{3}}C_{3}B_{3}^{\dag}-A_{4}V_{2}R_{B_{3}})B_{4}^{\dag}+A_{4}A_{3}T_{4}+T_{5}R_{B_{4}},
\end{align}
\begin{align}\label{equhh320}
X_{3}=A_{2}^{\dag}C_{2}+U_{1}B_{2}+L_{A_{2}}W_{1},
~\mbox{or}~
X_{3}=A_{3}^{\dag}C_{3}+U_{2}B_{3}+L_{A_{3}}T_{1},
\end{align}
where
\begin{align}
U_{1}=B_{11}^{\dag}D_{11}C_{11}^{\dag}+L_{B_{11}}W_{2}+W_{3}R_{C_{11}},
\end{align}
\begin{align}
V_{1}=-R_{A_{1}}(C_{1}-R_{A_{2}}C_{2}B_{2}^{\dag}B_{1}+A_{2}U_{1}B_{1})A_{11}^{\dag}+A_{1}W_{4}+W_{5}R_{A_{11}},
\end{align}
\begin{align}
V_{2}=A_{22}^{\dag}C_{22}B_{22}^{\dag}+L_{A_{22}}T_{2}+T_{3}R_{B_{22}},
\end{align}
\begin{align}
U_{2}=(A_{4}A_{3})^{\dag}(C_{4}+A_{4}R_{A_{3}}C_{3}B_{3}^{\dag}-A_{4}V_{2}R_{B_{3}})+T_{4}B_{4}+L_{(A_{4}A_{3})}T_{6},
\end{align}
\begin{align}
W_{1}=(I_{p_{1}},~0)[A_{33}^{\dag}(E_{1}-L_{B_{11}}W_{2}B_{2}-A_{44}T_{6}B_{3})-A_{33}^{\dag}Z_{7}B_{33}+L_{A_{33}}Z_{6}],
\end{align}
\begin{align}
T_{1}=(0,~I_{p_{2}})[A_{33}^{\dag}(E_{1}-L_{B_{11}}W_{2}B_{2}-A_{44}T_{6}B_{3})-A_{33}^{\dag}Z_{7}B_{33}+L_{A_{33}}Z_{6}],
\end{align}
\begin{align}
W_{3}=[R_{A_{33}}(E_{1}-L_{B_{11}}W_{2}B_{2}-A_{44}T_{6}B_{3})B_{33}^{\dag}
+A_{33}A_{33}^{\dag}Z_{7}+Z_{8}R_{B_{33}}]\begin{pmatrix}I_{p_{3}}\\0\end{pmatrix},
\end{align}
\begin{align}
T_{4}=[R_{A_{33}}(E_{1}-L_{B_{11}}W_{2}B_{2}-A_{44}T_{6}B_{3})B_{33}^{\dag}
+A_{33}A_{33}^{\dag}Z_{7}+Z_{8}R_{B_{33}}]\begin{pmatrix}0\\I_{p_{4}}\end{pmatrix},
\end{align}
\begin{align}
W_{2}=
A^{\dag}EB^{\dag}-A^{\dag}CM^{\dag}EB^{\dag}-A^{\dag}SC^{\dag
}EN^{\dag}DB^{\dag}-A^{\dag}SZ_{1}R_{N}DB^{\dag}+L_{A}Z_{2}+Z_{3}%
R_{B},
\end{align}
\begin{align}
T_{6}=M^{\dag}ED^{\dag}+S^{\dag}SC^{\dag}EN^{\dag}+L_{M}L_{S}Z_{4}%
+L_{M}Z_{1}R_{N}+Z_{5}R_{D},
\end{align}
the remaining $W_{j},T_{j},Z_{j}$ are  arbitrary matrices over $\mathbb{H}$, $p_{1}$ and $p_{2}$ are the column numbers of $A_{2}$ and $A_{3}$, respectively, $p_{3}$ and $p_{4}$ are the row numbers of $B_{1}$ and $B_{4}$, respectively.

\end{theorem}

\begin{proof}
We separate this system of one-sided coupled Sylvester-type real quaternion matrix equations (\ref{sys01}) into two parts
\begin{align} \label{equ38}
  \left\{\begin{array}{c}
A_{2}X_{3}-X_{2}B_{2}=C_{2},\\
A_{1}X_{1}-X_{2}B_{1}=C_{1},
\end{array}
  \right.
\end{align}
and
\begin{align} \label{equ39}
  \left\{\begin{array}{c}
A_{3}X_{3}-X_{4}B_{3}=C_{3},\\
A_{4}X_{4}-X_{5}B_{4}=C_{4}.
\end{array}
  \right.
\end{align}
Observe that system (\ref{equ38}) has the form of (\ref{bufensys02}), and system (\ref{equ39}) has the form of (\ref{bufensys01}). We can solve the system of one-sided coupled Sylvester-type real quaternion matrix equations (\ref{sys01}) through the following three steps. In the first step, we consider the system (\ref{equ38}). It follows from Lemma \ref{lemma01} that the system (\ref{equ38}) is consistent if and only if
\begin{align}
r\begin{pmatrix}C_{n}&A_{n}\\B_{n}&0\end{pmatrix}=r(A_{n})+r(B_{n}),~(n=1,2),
\end{align}
\begin{align}
r\begin{pmatrix}C_{1}&C_{2}&A_{1}&A_{2}\\B_{1}&B_{2}&0&0\end{pmatrix}=r(A_{1},~A_{2})+r(B_{1},~B_{2}),
\end{align}
or
\begin{align}
R_{A_{2}}C_{2}L_{B_{2}}=0,~R_{B_{11}}D_{11}=0,~D_{11}L_{C_{11}}=0.
\end{align}
In this case, the general solution to the system (\ref{equ38}) can be expressed as
\begin{align}\label{equ314}
X_{3}=A_{2}^{\dag}C_{2}+U_{1}B_{2}+L_{A_{2}}W_{1},
\end{align}
\begin{align}
X_{2}=-R_{A_{2}}C_{2}B_{2}^{\dag}+A_{2}U_{1}+V_{1}R_{B_{2}},
\end{align}
\begin{align}
X_{1}=A_{1}^{\dag}(C_{1}-R_{A_{2}}C_{2}B_{2}^{\dag}B_{1}+A_{2}U_{1}B_{1})+W_{4}A_{11}+L_{A_{1}}W_{6},
\end{align}
where
\begin{align}
U_{1}=B_{11}^{\dag}D_{11}C_{11}^{\dag}+L_{B_{11}}W_{2}+W_{3}R_{C_{11}},
\end{align}
\begin{align}
V_{1}=-R_{A_{1}}(C_{1}-R_{A_{2}}C_{2}B_{2}^{\dag}B_{1}+A_{2}U_{1}B_{1})A_{11}^{\dag}+A_{1}W_{4}+W_{5}R_{A_{11}},
\end{align}
and $W_{1}, \cdots, W_{6}$ are arbitrary matrices over $\mathbb{H}$ with appropriate sizes.

In the second step, we consider the system (\ref{equ39}). It follows from Lemma \ref{lemma02} that the system (\ref{equ39}) is consistent if and only if
\begin{align}
r\begin{pmatrix}C_{n}&A_{n}\\B_{n}&0\end{pmatrix}=r(A_{n})+r(B_{n}),~(n=3,4),
\end{align}
\begin{align}
r\begin{pmatrix}A_{4}C_{3}+C_{4}B_{3}&A_{4}A_{3}\\B_{4}B_{3}&0\end{pmatrix}=r(A_{4}A_{3})+r(B_{4}B_{3}),
\end{align}
or
\begin{align}
R_{A_{3}}C_{3}L_{B_{3}}=0,~R_{A_{22}}C_{22}=0,~C_{22}L_{B_{22}}=0.
\end{align}
In this case, the general solution to the system (\ref{equ39}) can be expressed as
\begin{align}\label{equ323}
X_{3}=A_{3}^{\dag}C_{3}+U_{2}B_{3}+L_{A_{3}}T_{1},
\end{align}
\begin{align}
X_{4}=-R_{A_{3}}C_{3}B_{3}^{\dag}+A_{3}U_{2}+V_{2}R_{B_{3}},
\end{align}
\begin{align}
X_{5}=-R_{(A_{4}A_{3})}(C_{4}+A_{4}R_{A_{3}}C_{3}B_{3}^{\dag}-A_{4}V_{2}R_{B_{3}})B_{4}^{\dag}+A_{4}A_{3}T_{4}+T_{5}R_{B_{4}},
\end{align}
where
\begin{align}
V_{2}=A_{22}^{\dag}C_{22}B_{22}^{\dag}+L_{A_{22}}T_{2}+T_{3}R_{B_{22}},
\end{align}
\begin{align}
U_{2}=(A_{4}A_{3})^{\dag}(C_{4}+A_{4}R_{A_{3}}C_{3}B_{3}^{\dag}-A_{4}V_{2}R_{B_{3}})+T_{4}B_{4}+L_{(A_{4}A_{3})}T_{6},
\end{align}
and $T_{1}, \cdots, T_{6}$ are arbitrary matrices over $\mathbb{H}$ with appropriate sizes.

In the third step, equating $X_{3}$ in (\ref{equ314}) and $X_{3}$ in (\ref{equ323}) gives
\begin{align*}
&A_{2}^{\dag}C_{2}+(B_{11}^{\dag}D_{11}C_{11}^{\dag}+L_{B_{11}}W_{2}+W_{3}R_{C_{11}})B_{2}+L_{A_{2}}W_{1}\\=
&A_{3}^{\dag}C_{3}+(A_{4}A_{3})^{\dag}(C_{4}+A_{4}R_{A_{3}}C_{3}B_{3}^{\dag})B_{3}+T_{4}B_{4}B_{3}+L_{(A_{4}A_{3})}T_{6}B_{3}+L_{A_{3}}T_{1},
\end{align*}
i.e.,
\begin{align} \label{equ328}
A_{33}\begin{pmatrix}W_{1}\\T_{1}\end{pmatrix}+(W_{3},~T_{4})B_{33}+L_{B_{11}}W_{2}B_{2}+A_{44}T_{6}B_{3}=E_{1}.
\end{align}
Now we want to solve the matrix equation (\ref{equ328}). It follows from Lemma \ref{lemma03} that the matrix equation (\ref{equ328}) is consistent if and only if
\begin{align}\label{equ331}
R_{M}R_{A}E=0,EL_{B}L_{N}=0,R_{A}EL_{D}=0,R_{C}EL_{B}=0.
\end{align}
Hence, the general solution to the matrix equation (\ref{equ328}) can be expressed as
\begin{align}
\begin{pmatrix}W_{1}\\T_{1}\end{pmatrix}=
A_{33}^{\dag}(E_{1}-L_{B_{11}}W_{2}B_{2}-A_{44}T_{6}B_{3})-A_{33}^{\dag}Z_{7}B_{33}+L_{A_{33}}Z_{6},
\end{align}
\begin{align}
(W_{3},~T_{4})=R_{A_{33}}(E_{1}-L_{B_{11}}W_{2}B_{2}-A_{44}T_{6}B_{3})B_{33}^{\dag}
+A_{33}A_{33}^{\dag}Z_{7}+Z_{8}R_{B_{33}},
\end{align}
\begin{align}
W_{2}=
A^{\dag}EB^{\dag}-A^{\dag}CM^{\dag}EB^{\dag}-A^{\dag}SC^{\dag
}EN^{\dag}DB^{\dag}-A^{\dag}SZ_{1}R_{N}DB^{\dag}+L_{A}Z_{2}+Z_{3}%
R_{B},
\end{align}
\begin{align}
T_{6}=M^{\dag}ED^{\dag}+S^{\dag}SC^{\dag}EN^{\dag}+L_{M}L_{S}Z_{4}%
+L_{M}Z_{1}R_{N}+Z_{5}R_{D},
\end{align}
where $Z_{1},\ldots,Z_{8}$ are arbitrary matrices over $\mathbb{H}$ with appropriate
sizes.

Now we want to prove that (\ref{equ331}) $\Longleftrightarrow$ (\ref{equ034})-(\ref{equ037}). At first, we prove that $R_{M}R_{A}E=0$ is equivalent to (\ref{equ034}). Applying Lemma \ref{lemma04} to $R_{M}R_{A}E=0$ gives
\begin{align*}
&R_{M}R_{A}E=0
\Leftrightarrow  r(R_{A}E,~M)=r(M)
\\&
\Leftrightarrow r(R_{A}E,~R_{A}C)=r(R_{A}C)\\&
\Leftrightarrow r(A,~C,~E)=r(A,~C)\\&
\Leftrightarrow r\begin{pmatrix}E_{1}&L_{B_{11}}&A_{44}&A_{33}\\B_{33}&0&0&0\end{pmatrix}=r(L_{B_{11}},~A_{44},~A_{33})+r(B_{33})\\
&\Leftrightarrow r\begin{pmatrix}E_{1}&L_{B_{11}}&L_{(A_{4}A_{3})}&L_{A_{2}}&L_{A_{3}}\\
R_{C_{11}}B_{2}&0&0&0&0\\B_{4}B_{3}&0&0&0&0\end{pmatrix}=r(L_{B_{11}},L_{(A_{4}A_{3})},L_{A_{2}},L_{A_{3}})+
r\begin{pmatrix}R_{C_{11}}B_{2}\\B_{4}B_{3}\end{pmatrix}\\
&\Leftrightarrow r\begin{pmatrix}E_{1}&I&L_{(A_{4}A_{3})}\\R_{C_{11}}B_{2}&0&0\\B_{4}B_{3}&0&0\\0&B_{11}&0\end{pmatrix}=
r\begin{pmatrix}I&L_{(A_{4}A_{3})}\\B_{11}&0\end{pmatrix}+r\begin{pmatrix}R_{C_{11}}B_{2}\\B_{4}B_{3}\end{pmatrix}\\
&\Leftrightarrow r\begin{pmatrix}E_{1}&I&I&0&0\\B_{2}&0&0&0&B_{1}\\B_{4}B_{3}&0&0&0&0\\0&A_{2}&0&A_{1}&0\\0&0&A_{4}A_{3}&0&0\end{pmatrix}=
r\begin{pmatrix}I&I&0\\A_{2}&0&A_{1}\\0&A_{4}A_{3}&0\end{pmatrix}+r\begin{pmatrix}B_{1}&B_{2}\\0&B_{4}B_{3}\end{pmatrix}\\
&\Leftrightarrow r\begin{pmatrix}C_{1}&C_{2}&A_{1}&A_{2}\\0&A_{4}C_{3}+C_{4}B_{3}&0&A_{4}A_{3}\\B_{1}&B_{2}&0&0\\0&B_{4}B_{3}&0&0\end{pmatrix}
=r\begin{pmatrix}A_{1}&A_{2}\\0&A_{4}A_{3}\end{pmatrix}+r\begin{pmatrix}B_{1}&B_{2}\\0&B_{4}B_{3}\end{pmatrix}\\
&\Leftrightarrow (\ref{equ034}).
\end{align*}Similarly, we can show that $EL_{B}L_{N}=0,R_{A}EL_{D}=0$ and $R_{C}EL_{B}=0$ are equivalent to (\ref{equ035}), (\ref{equ036}) and (\ref{equ037}), respectively.
\end{proof}

Next we give an example to illustrate Theorem \ref{theorem01}.

\begin{example}
Given the quaternion matrices:
\begin{align*}
A_{1}=\begin{pmatrix}1+\mathbf{k}&1+\mathbf{i}-\mathbf{k}&\mathbf{i}+\mathbf{j}\\
-1&2\mathbf{k}&2+\mathbf{j}+\mathbf{k}\\ \mathbf{k}& 1+\mathbf{i}+\mathbf{k}&2+\mathbf{i}+2\mathbf{j}+\mathbf{k}\end{pmatrix}, ~
B_{1}=\begin{pmatrix}2\mathbf{i}+\mathbf{k}&0&1+\mathbf{i}+\mathbf{j}+\mathbf{k}\\
2\mathbf{i}-\mathbf{j}+\mathbf{k}&-1+\mathbf{j}&1+\mathbf{j}\\ \mathbf{j}&1-\mathbf{j}& \mathbf{i}+\mathbf{k}\end{pmatrix},
\end{align*}
\begin{align*}
A_{2}=\begin{pmatrix}1+\mathbf{k}&2&\mathbf{i}\\2\mathbf{j}&1-\mathbf{j}&-\mathbf{i}+\mathbf{k}\\
\mathbf{i}+\mathbf{j}+\mathbf{k}&1&\mathbf{k}\end{pmatrix},~
B_{2}=\begin{pmatrix}-2+\mathbf{k}&\mathbf{i}&1-\mathbf{j}\\-\mathbf{j}&1&\mathbf{i}-\mathbf{k}\\
-2&1+\mathbf{i}+\mathbf{j}&0\end{pmatrix},
\end{align*}
\begin{align*}
A_{3}=\begin{pmatrix}1+\mathbf{k}&\mathbf{i}+\mathbf{k}&\mathbf{j}\\ \mathbf{i}-\mathbf{j}&-1-\mathbf{j}&\mathbf{k}\\
1+\mathbf{i}-\mathbf{j}+\mathbf{k}&-1+\mathbf{i}-\mathbf{j}+\mathbf{k}&\mathbf{j}+\mathbf{k}\end{pmatrix},~
B_{3}=\begin{pmatrix}-1+\mathbf{j}+\mathbf{k}&1+\mathbf{k}&\mathbf{i}+\mathbf{j}\\
1-\mathbf{j}+\mathbf{k}&-1-\mathbf{j}&-\mathbf{i}+\mathbf{k}\\
2\mathbf{k}&-\mathbf{j}+\mathbf{k}&\mathbf{j}+\mathbf{k}\end{pmatrix},
\end{align*}
\begin{align*}
A_{4}=\begin{pmatrix}\mathbf{j}&1-\mathbf{j}&\mathbf{i}+\mathbf{k}\\
\mathbf{i}&-1+\mathbf{k}&\mathbf{j}\\
\mathbf{i}+\mathbf{j}&-\mathbf{j}+\mathbf{k}&\mathbf{i}+\mathbf{j}+\mathbf{k}\end{pmatrix},~
B_{4}=\begin{pmatrix}1+\mathbf{k}&-1-\mathbf{k}&\mathbf{i}+\mathbf{j}\\
\mathbf{j}&\mathbf{i}+\mathbf{k}&1\\
1+\mathbf{j}+\mathbf{k}&-1+\mathbf{i}&1+\mathbf{i}+\mathbf{j}\end{pmatrix},
\end{align*}
\begin{align*}
C_{1}=\begin{pmatrix}-1-\mathbf{i}+5\mathbf{k}&-2-3\mathbf{j}+5\mathbf{k}&3\mathbf{i}+5\mathbf{j}\\
-1-7\mathbf{i}+\mathbf{j}+2\mathbf{k}&3\mathbf{j}+5\mathbf{k}&4-8\mathbf{i}+7\mathbf{j}\\
-5-3\mathbf{i}+\mathbf{k}&-4-3\mathbf{i}+2\mathbf{j}+5\mathbf{k}&1-5\mathbf{i}+9\mathbf{j}-4\mathbf{k}\end{pmatrix},
\end{align*}
\begin{align*}
C_{2}=\begin{pmatrix}4\mathbf{i}+6\mathbf{j}+3\mathbf{k}&-2+5\mathbf{i}+4\mathbf{j}&-3-\mathbf{i}-\mathbf{j}\\
2-\mathbf{i}+9\mathbf{j}-6\mathbf{k}&6-5\mathbf{i}+5\mathbf{j}-4\mathbf{k}&-5+2\mathbf{i}-3\mathbf{j}+2\mathbf{k}\\
4+\mathbf{i}+6\mathbf{j}+3\mathbf{k}&-2-\mathbf{i}+3\mathbf{j}-2\mathbf{k}&-2-\mathbf{i}+3\mathbf{j}\end{pmatrix},
\end{align*}
\begin{align*}
C_{3}=\begin{pmatrix}7\mathbf{j}+6\mathbf{k}&4+\mathbf{i}+\mathbf{j}&-1+2\mathbf{i}+6\mathbf{j}\\
-3-5\mathbf{i}+2\mathbf{j}+\mathbf{k}&-1+6\mathbf{j}+5\mathbf{k}&2-\mathbf{i}+4\mathbf{j}\\
-8+2\mathbf{i}+3\mathbf{j}+6\mathbf{k}&3+4\mathbf{i}+9\mathbf{j}-2\mathbf{k}&-2+\mathbf{i}+3\mathbf{j}-2\mathbf{k}\end{pmatrix},
\end{align*}
\begin{align*}
C_{4}=\begin{pmatrix}-1-2\mathbf{i}-\mathbf{j}-3\mathbf{k}&3-3\mathbf{i}-2\mathbf{j}+\mathbf{k}&1+2\mathbf{i}\\
2-3\mathbf{i}-3\mathbf{j}-\mathbf{k}&-5-2\mathbf{i}+\mathbf{j}+\mathbf{k}&1-2\mathbf{i}-2\mathbf{j}\\
1-5\mathbf{i}-4\mathbf{j}-4\mathbf{k}&-2-5\mathbf{i}-\mathbf{j}+2\mathbf{k}&2-2\mathbf{j}\end{pmatrix}.
\end{align*}
Now we consider the system of one-sided coupled Sylvester-type real quaternion matrix equations (\ref{sys01}). Check that
\begin{align*}
r\begin{pmatrix}C_{i}&A_{i}\\B_{i}&0\end{pmatrix}=r(A_{i})+r(B_{i})=
\left\{\begin{array}{lll}
4, &\mbox{if}~
 i=1\\
6, & \mbox{if}~
 i=2\\
3, &\mbox{if}~
 i=3\\
4, & \mbox{if}~
 i=4
\end{array},
 \right.
\end{align*}
\begin{align*}
r\begin{pmatrix}C_{1}&C_{2}&A_{1}&A_{2}\\B_{1}&B_{2}&0&0\end{pmatrix}=r(A_{1},~A_{2})+r(B_{1},~B_{2})=6,
\end{align*}
\begin{align*}
r\begin{pmatrix}A_{4}C_{3}+C_{4}B_{3}&A_{4}A_{3}\\B_{4}B_{3}&0\end{pmatrix}=r(A_{4}A_{3})+r(B_{4}B_{3})=3,
\end{align*}
\begin{align*}
r\begin{pmatrix}C_{1}&C_{2}&A_{1}&A_{2}\\0&A_{4}C_{3}+C_{4}B_{3}&0&A_{4}A_{3}\\B_{1}&B_{2}&0&0\\0&B_{4}B_{3}&0&0\end{pmatrix}
=r\begin{pmatrix}A_{1}&A_{2}\\0&A_{4}A_{3}\end{pmatrix}+r\begin{pmatrix}B_{1}&B_{2}\\0&B_{4}B_{3}\end{pmatrix}=9,
\end{align*}
\begin{align*}
r\begin{pmatrix}C_{2}&A_{2}\\C_{3}&A_{3}\\B_{2}&0\\B_{3}&0\end{pmatrix}
=r\begin{pmatrix}A_{2}\\A_{3}\end{pmatrix}+r\begin{pmatrix}B_{2}\\B_{3}\end{pmatrix}=6,
\end{align*}
\begin{align*}
r\begin{pmatrix}C_{1}&C_{2}&A_{1}&A_{2}\\C_{3}&0&0&A_{3}\\B_{1}&B_{2}&0&0\\0&B_{3}&0&0\end{pmatrix}=
r\begin{pmatrix}A_{1}&A_{2}\\0&A_{3}\end{pmatrix}+r\begin{pmatrix}B_{1}&B_{2}\\0&B_{3}\end{pmatrix}=9.
\end{align*}
\begin{align*}
r\begin{pmatrix}C_{2}&A_{2}\\A_{4}C_{3}+C_{4}B_{3}&A_{4}A_{3}\\B_{2}&0\\B_{4}B_{3}&0\end{pmatrix}=
r\begin{pmatrix}A_{2}\\A_{4}A_{3}\end{pmatrix}+r\begin{pmatrix}B_{2}\\B_{4}B_{3}\end{pmatrix}=6.
\end{align*}
All the rank equalities in (\ref{equhh037})-(\ref{equ037}) hold. Hence, the system of one-sided coupled Sylvester-type real quaternion matrix equations (\ref{sys01}) is consistent. Note that
\begin{align*}
X_{1}=\begin{pmatrix}2+\mathbf{j}+\mathbf{k}&1-2\mathbf{i}+\mathbf{k}&\mathbf{i}-2\mathbf{k}\\
-2&\mathbf{i}+\mathbf{j}&1+\mathbf{i}+2\mathbf{j}\\
\mathbf{j}+\mathbf{k}&1-\mathbf{i}+\mathbf{j}+\mathbf{k}&1+2\mathbf{i}+2\mathbf{j}-2\mathbf{k}\end{pmatrix}
~
X_{2}=\begin{pmatrix}\mathbf{i}+\mathbf{j}&-\mathbf{i}-\mathbf{j}&\mathbf{k}\\
1+2\mathbf{j}&\mathbf{i}+\mathbf{j}+\mathbf{k}&-1+\mathbf{i}+\mathbf{j}-\mathbf{k}\\
2+\mathbf{k}&-\mathbf{k}&1+\mathbf{j}+2\mathbf{k}\end{pmatrix},
\end{align*}
\begin{align*}
X_{3}=\begin{pmatrix}2\mathbf{i}+\mathbf{k}&1+3\mathbf{i}-\mathbf{k}&\mathbf{j}\\
1+\mathbf{j}&-1&\mathbf{i}-\mathbf{j}\\
1+2\mathbf{i}+\mathbf{j}+\mathbf{k}&3\mathbf{i}-\mathbf{k}&\mathbf{i}\end{pmatrix}
~
X_{4}=\begin{pmatrix}-2+\mathbf{j}&-1+2\mathbf{i}&\mathbf{k}\\
0&1+\mathbf{k}&\mathbf{i}+\mathbf{j}\\
-2+\mathbf{k}&\mathbf{j}-\mathbf{k}&1\end{pmatrix},
\end{align*}and
\begin{align*}
X_{5}=\begin{pmatrix}-2+\mathbf{j}&-1+2\mathbf{i}&\mathbf{k}\\
0&1+\mathbf{k}&\mathbf{i}+\mathbf{j}\\
-2+\mathbf{k}&\mathbf{j}-\mathbf{k}&1\end{pmatrix}
\end{align*}
satisfy the system (\ref{sys01}).

\end{example}

Let $A_{4},B_{4},$ and $C_{4}$ vanish in Theorem \ref{theorem01}. Then we can obtain some necessary and sufficient conditions and general solution to the system of coupled generalized  Sylvester real quaternion matrix
equations (\ref{specialsys01}).
\begin{corollary}\label{coro01}\cite{auto001}
Let $A_{i},B_{i},$ and $C_{i}(i=1,2,3)$ be given. Set
\begin{align*}
&A_{4}=A_{2}L_{A_{3}},B_{4}=R_{B_{1}}B_{2},A=R_{A_{4}}A_{2},B=B_{3}L_{B_{4}},\\&C=R_{A_{4}}A_{1},D=B_{2}L_{B_{4}}
,M=R_{A}C,N=DL_{B},S=CL_{M},\\&C_{4}=C_{2}-A_{2}A_{3}^{\dag}C_{3}-R_{A_{1}}C_{1}B_{1}^{\dag}B_{2},E=R_{A_{4}}C_{4}L_{B_{4}}.
\end{align*}
Then the following statements are equivalent:\\
$(1)$ The system of coupled generalized  Sylvester real quaternion matrix
equations (\ref{specialsys01}) is consistent.\\
$(2)$
\begin{align*}
r\begin{pmatrix}C_{i}&A_{i}\\B_{i}&0\end{pmatrix}=r(A_{i})+r(B_{i}),(i=1,2,3),
\end{align*}
\begin{align*}
r\begin{pmatrix}A_{1}&A_{2}&C_{1}&C_{2}\\0&0&B_{1}&B_{2}\end{pmatrix}=r(A_{1}~A_{2})
+r(B_{1}~B_{2}),
\end{align*}
\begin{align*}
r\begin{pmatrix}B_{2}&0\\B_{3}&0\\C_{2}&A_{2}\\C_{3}&A_{3}\end{pmatrix}
=r\begin{pmatrix}A_{2}\\A_{3}\end{pmatrix}+r\begin{pmatrix}B_{2}\\B_{3}\end{pmatrix},
\end{align*}
\begin{align*}
r\begin{pmatrix}C_{2}&C_{1}&A_{1}&A_{2}\\C_{3}&0&0&A_{3}\\B_{2}&B_{1}&0&0\\B_{3}&0&0&0\end{pmatrix}
=r\begin{pmatrix}A_{1}&A_{2}\\0&A_{3}\end{pmatrix}+r\begin{pmatrix}B_{2}&B_{1}\\B_{3}&0\end{pmatrix}.
\end{align*}

$(3)$
\begin{align*}
&R_{A_{i}}C_{i}L_{B_{i}}=0,(i=1,2),R_{M}R_{A}E=0,\\&
EL_{B}L_{N}=0, R_{A}EL_{D}=0,  R_{C}EL_{B}=0.
\end{align*}

In this case, the general solution to the coupled generalized  Sylvester real quaternion matrix
equations (\ref{specialsys01}) can be expressed as
\begin{align*}
X=A_{1}^{\dag}C_{1}-U_{1}B_{1}-L_{A_{1}}U_{2},
~
Y=-R_{A_{1}}C_{1}B_{1}^{\dag}-A_{1}U_{1}-U_{3}R_{B_{1}},
\end{align*}
\begin{align*}
Z=A_{3}^{\dag}C_{3}+V_{1}B_{3}+L_{A_{3}}V_{2},
~
W=-R_{A_{3}}C_{3}B_{3}^{\dag}+A_{3}V_{1}+V_{3}R_{B_{3}},
\end{align*}
where
\begin{align*}
  V_{2}=A_{4}^{\dag}(C_{4}-A_{2}V_{1}B_{3}-A_{1}U_{1}B_{2})-A_{4}^{\dag}%
T_{7}B_{4}+L_{A_{4}}T_{6},
\end{align*}
\begin{align*}
  U_{3}=R_{A_{4}}(C_{4}-A_{2}V_{1}B_{3}-A_{1}U_{1}B_{2})B_{4}^{\dag}%
+A_{4}A_{4}^{\dag}T_{7}+T_{8}R_{B_{4}},
\end{align*}
\begin{align*}
V_{1}=A^{\dag}EB^{\dag}-A^{\dag}CM^{\dag}EB^{\dag}-A^{\dag}SC^{\dag
}EN^{\dag}DB^{\dag}-A^{\dag}ST_{2}R_{N}DB^{\dag}+L_{A}T_{4}+T_{5}%
R_{B},
\end{align*}
\begin{align*}
 U_{1}=M^{\dag}ED^{\dag}+S^{\dag}SC^{\dag}EN^{\dag}+L_{M}L_{S}T_{1}%
+L_{M}T_{2}R_{N}+T_{3}R_{D},
\end{align*}
and $U_{2},V_{3},T_{1}, \ldots, T_{8}$ are arbitrary matrices over
$\mathbb{H}$ with appropriate sizes.
\end{corollary}

\section{\textbf{Some solvability conditions and the general solution to
system (\ref{sys02})}}

In this section, we consider the solvability conditions and the general solution to the
system of one-sided coupled Sylvester-type real quaternion matrix equations (\ref{sys02}). For simplicity, put
\begin{align*}
A_{11}=R_{(A_{2}A_{1})}A_{2},~B_{11}=R_{B_{1}}L_{B_{2}},~C_{11}=R_{(A_{2}A_{1})}(A_{2}R_{A_{1}}C_{1}B_{1}^{\dag}+C_{2})L_{B_{2}},
\end{align*}
\begin{align*}
A_{22}=R_{B_{3}}B_{4},~B_{22}=R_{A_{4}}A_{3},~C_{22}=B_{4}L_{A_{22}},~D_{22}=R_{A_{4}}(R_{A_{3}}C_{3}B_{3}^{\dag}B_{4}-C_{4})L_{A_{22}},
\end{align*}
\begin{align*}
A_{33}=(A_{2}A_{1},~-L_{A_{3}}),~B_{33}=\begin{pmatrix}R_{B_{2}}\\-R_{C_{22}}B_{3}\end{pmatrix},~A_{44}=R_{B_{11}}R_{B_{1}}B_{2}^{\dag},~B_{44}=-L_{B_{22}},
\end{align*}
\begin{align*}
E_{1}=A_{3}^{\dag}C_{3}+B_{22}^{\dag}D_{22}C_{22}^{\dag}B_{3}+R_{(A_{2}A_{1})}C_{2}B_{2}^{\dag}+A_{11}R_{A_{1}}C_{1}B_{1}^{\dag}B_{2}^{\dag}-C_{11}B_{11}^{\dag}R_{B_{1}}B_{2}^{\dag},
\end{align*}
\begin{align*}
A=R_{A_{33}}A_{11},~B=A_{44}L_{B_{33}},~C=R_{A_{33}}B_{44},~D=B_{3}L_{B_{33}},
\end{align*}
\begin{align*}
E=R_{A_{33}}E_{1}L_{B_{33}},M=R_{A}C,N=DL_{B},S=CL_{M}.
\end{align*}

Now we give the fundamental theorem of this section.
\begin{theorem}\label{theorem02}
Let $A_{i},B_{i},$ and $C_{i}(i=1,2,3,4)$ be given. Then the following statements are equivalent:\\$(1)$ The system of one-sided coupled Sylvester-type real quaternion matrix equations (\ref{sys02}) is consistent.\\$(2)$
\begin{align}\label{equhh047}
r\begin{pmatrix}C_{i}&A_{i}\\B_{i}&0\end{pmatrix}=r(A_{i})+r(B_{i}),\quad (i=1,2,3,4),
\end{align}
\begin{align}
r\begin{pmatrix}A_{2}C_{1}+C_{2}B_{1}&A_{2}A_{1}\\B_{2}B_{1}&0\end{pmatrix}=r(A_{2}A_{1})+r(B_{2}B_{1}),
\end{align}
\begin{align}
r\begin{pmatrix}C_{3}&C_{4}&A_{3}&A_{4}\\B_{3}&B_{4}&0&0\end{pmatrix}=r(A_{3},~A_{4})+r(B_{3},~B_{4}),
\end{align}
\begin{align}
r\begin{pmatrix}A_{3}C_{2}+C_{3}B_{2}&C_{4}&A_{3}A_{2}&A_{4}\\B_{3}B_{2}&B_{4}&0&0\end{pmatrix}=r(A_{3}A_{2},~A_{4})+r(B_{3}B_{2},~B_{4}),
\end{align}
\begin{align}
r\begin{pmatrix}A_{3}A_{2}C_{1}+A_{3}C_{2}B_{1}+C_{3}B_{2}B_{1}&A_{3}A_{2}A_{1}\\B_{3}B_{2}B_{1}&0\end{pmatrix}
=r(A_{3}A_{2}A_{1})
+r(B_{3}B_{2}B_{1}),
\end{align}
\begin{align}
r\begin{pmatrix}A_{3}C_{2}+C_{3}B_{2}&A_{3}A_{2}\\B_{3}B_{2}&0\end{pmatrix}=r(A_{3}A_{2})+r(B_{3}B_{2}),
\end{align}
\begin{align}\label{equhh413}
r\begin{pmatrix}A_{3}A_{2}C_{1}+A_{3}C_{2}B_{1}+C_{3}B_{2}B_{1}&C_{4}&A_{4}&A_{3}A_{2}A_{1}\\B_{3}B_{2}B_{1}&B_{4}&0&0\end{pmatrix}
=r(A_{3}A_{2}A_{1},~A_{4})+r(B_{3}B_{2}B_{1},~B_{4}).
\end{align}
$(3)$
\begin{align*}
R_{A_{1}}C_{1}L_{B_{1}}=0,~R_{A_{11}}C_{11}=0,~C_{11}L_{B_{11}}=0,
\end{align*}
\begin{align*}
R_{A_{3}}C_{3}L_{B_{3}}=0,~R_{B_{22}}D_{22}=0,~D_{22}L_{C_{22}}=0,
\end{align*}
\begin{align*}
R_{M}R_{A}E=0,~EL_{B}L_{N}=0,~R_{A}EL_{D}=0,~R_{C}EL_{B}=0.
\end{align*}

In this case, the general solution to the system of one-sided coupled Sylvester-type real quaternion matrix equations (\ref{sys02}) can be expressed as
\begin{align*}
X_{1}=A_{1}^{\dag}C_{1}+U_{1}B_{1}+L_{A_{1}}W_{1},~
X_{2}=-R_{A_{1}}C_{1}B_{1}^{\dag}+A_{1}U_{1}+V_{1}R_{B_{1}},
\end{align*}
\begin{align*}
X_{4}=-R_{A_{3}}C_{3}B_{3}^{\dag}+A_{3}U_{2}+V_{2}R_{B_{3}},
\end{align*}
\begin{align*}
X_{5}=A_{4}^{\dag}(C_{4}-R_{A_{3}}C_{3}B_{3}^{\dag}B_{4}+A_{3}U_{1}B_{4})+T_{4}A_{22}+L_{A_{4}}T_{6},
\end{align*}
\begin{align*}
X_{3}=-R_{(A_{2}A_{1})}(C_{2}+A_{2}R_{A_{1}}C_{1}B_{1}^{\dag}-A_{2}V_{1}R_{B_{1}})B_{2}^{\dag}+A_{2}A_{1}W_{4}+W_{5}R_{B_{2}},
\end{align*}
or
\begin{align*}
X_{3}=A_{3}^{\dag}C_{3}+U_{2}B_{3}+L_{A_{3}}T_{1},
\end{align*}
where
\begin{align*}
V_{1}=A_{11}^{\dag}C_{11}B_{11}^{\dag}+L_{A_{11}}W_{2}+W_{3}R_{B_{11}},
\end{align*}
\begin{align*}
U_{1}=(A_{2}A_{1})^{\dag}(C_{2}+A_{2}R_{A_{1}}C_{1}B_{1}^{\dag}-A_{2}V_{1}R_{B_{1}})+W_{4}B_{2}+L_{(A_{2}A_{1})}W_{6},
\end{align*}
\begin{align*}
U_{2}=B_{22}^{\dag}D_{22}C_{22}^{\dag}+L_{B_{22}}T_{2}+T_{3}R_{C_{22}},
\end{align*}
\begin{align*}
V_{2}=-R_{A_{4}}(C_{4}-R_{A_{3}}C_{3}B_{3}^{\dag}B_{4}+A_{3}U_{2}B_{4})A_{22}^{\dag}+A_{4}T_{4}+T_{5}R_{A_{22}},
\end{align*}
\begin{align*}
W_{4}=(I_{p_{1}},~0)[
A_{33}^{\dag}(E_{1}-A_{11}W_{3}A_{44}-B_{44}T_{2}B_{3})-A_{33}^{\dag}Z_{7}B_{33}+L_{A_{33}}Z_{6}],
\end{align*}
\begin{align*}
T_{1}=(0,~I_{p_{2}})[
A_{33}^{\dag}(E_{1}-A_{11}W_{3}A_{44}-B_{44}T_{2}B_{3})-A_{33}^{\dag}Z_{7}B_{33}+L_{A_{33}}Z_{6}],
\end{align*}
\begin{align*}
W_{5}=[R_{A_{33}}(E_{1}-A_{11}W_{3}A_{44}-B_{44}T_{2}B_{3})B_{33}^{\dag}
+A_{33}A_{33}^{\dag}Z_{7}+Z_{8}R_{B_{33}}]\begin{pmatrix}I_{p_{3}}\\0\end{pmatrix},
\end{align*}
\begin{align*}
T_{3}=[R_{A_{33}}(E_{1}-A_{11}W_{3}A_{44}-B_{44}T_{2}B_{3})B_{33}^{\dag}
+A_{33}A_{33}^{\dag}Z_{7}+Z_{8}R_{B_{33}}]\begin{pmatrix}0\\I_{p_{4}}\end{pmatrix},
\end{align*}
\begin{align*}
W_{3}=
A^{\dag}EB^{\dag}-A^{\dag}CM^{\dag}EB^{\dag}-A^{\dag}SC^{\dag
}EN^{\dag}DB^{\dag}-A^{\dag}SZ_{1}R_{N}DB^{\dag}+L_{A}Z_{2}+Z_{3}%
R_{B},
\end{align*}
\begin{align*}
T_{2}=M^{\dag}ED^{\dag}+S^{\dag}SC^{\dag}EN^{\dag}+L_{M}L_{S}Z_{4}%
+L_{M}Z_{1}R_{N}+Z_{5}R_{D},
\end{align*}
the remaining $W_{j},T_{j},Z_{j}$ are  arbitrary matrices over $\mathbb{H}$, $p_{1}$ and $p_{2}$ are the column numbers of $A_{1}$ and $A_{3}$, respectively, $p_{3}$ and $p_{4}$ are the row numbers of $B_{2}$ and $B_{4}$, respectively.

\end{theorem}

\begin{proof}
We separate this system of one-sided coupled Sylvester-type real quaternion matrix equations (\ref{sys02}) into two parts
\begin{align}
  \left\{\begin{array}{c}
A_{1}X_{1}-X_{2}B_{1}=C_{1},\\
A_{2}X_{2}-X_{3}B_{2}=C_{2},
\end{array}
  \right.
\end{align}
and
\begin{align}
  \left\{\begin{array}{c}
A_{3}X_{3}-X_{4}B_{3}=C_{3},\\
A_{4}X_{5}-X_{4}B_{4}=C_{4}.
\end{array}
  \right.
\end{align}
Applying the main idea of Theorem \ref{theorem01}, Lemma \ref{lemma01}, Lemma \ref{lemma02}, Lemma \ref{lemma03} and Lemma \ref{lemma04}, we can prove Theorem \ref{theorem02}.
\end{proof}
Now we give an example to illustrate Theorem \ref{theorem02}.

\begin{example}
Given the quaternion matrices:
\begin{align*}
A_{1}=\begin{pmatrix}\mathbf{i}&\mathbf{j}&1+\mathbf{k}\\
\mathbf{k}&\mathbf{i}+\mathbf{j}-2\mathbf{k}&-2+\mathbf{k}\\
1+\mathbf{i}+\mathbf{j}& 2-\mathbf{i}&-\mathbf{j}\end{pmatrix}, ~
B_{1}=\begin{pmatrix}1&\mathbf{k}&\mathbf{i}+\mathbf{k}\\
\mathbf{i}&\mathbf{k}&-1+\mathbf{k}\\
1+\mathbf{i}&2\mathbf{k}&-1+\mathbf{i}+2\mathbf{k}\end{pmatrix},
\end{align*}
\begin{align*}
A_{2}=\begin{pmatrix}1+\mathbf{i}+\mathbf{k}&1-\mathbf{i}&1+\mathbf{j}+\mathbf{k}\\
-1+\mathbf{i}-\mathbf{j}&1+\mathbf{i}&\mathbf{i}-\mathbf{j}+\mathbf{k}\\
2\mathbf{i}-\mathbf{j}+\mathbf{k}&2&1+\mathbf{i}+2\mathbf{k}\end{pmatrix},~
B_{2}=\begin{pmatrix}1+\mathbf{j}-\mathbf{k}&-1+2\mathbf{j}-\mathbf{k}&2\\
\mathbf{i}+\mathbf{j}&1+\mathbf{k}&1+\mathbf{i}-\mathbf{j}\\
1+\mathbf{i}+2\mathbf{j}-\mathbf{k}&2\mathbf{j}&3+\mathbf{i}-\mathbf{j}\end{pmatrix},
\end{align*}
\begin{align*}
A_{3}=\begin{pmatrix}2\mathbf{j}&1-\mathbf{j}&\mathbf{i}+\mathbf{k}\\
\mathbf{i}+\mathbf{j}+\mathbf{k}&1+\mathbf{j}&-1\\
\mathbf{i}+3\mathbf{j}+\mathbf{k}&2&-1+\mathbf{i}+\mathbf{k}\end{pmatrix},~
B_{3}=\begin{pmatrix}-1&1+\mathbf{j}&\mathbf{i}-\mathbf{j}\\
-\mathbf{i}&\mathbf{i}+2\mathbf{j}&-1-2\mathbf{j}\\
-1-\mathbf{i}&1+\mathbf{i}+\mathbf{k}&-1+\mathbf{i}-\mathbf{k}\end{pmatrix},
\end{align*}
\begin{align*}
A_{4}=\begin{pmatrix}2+\mathbf{i}+\mathbf{j}-2\mathbf{k}&1-2\mathbf{i}-\mathbf{j}&\mathbf{i}\\
1+\mathbf{j}-2\mathbf{k}&0&-\mathbf{i}+\mathbf{j}\\
3+\mathbf{i}&1-2\mathbf{i}+\mathbf{k}&0\end{pmatrix},~
B_{4}=\begin{pmatrix}\mathbf{j}&1-\mathbf{j}&\mathbf{i}+\mathbf{k}\\
\mathbf{i}&1+\mathbf{j}&-\mathbf{k}\\
\mathbf{i}+\mathbf{j}&2&\mathbf{i}\end{pmatrix},
\end{align*}
\begin{align*}
C_{1}=\begin{pmatrix}1+\mathbf{i}-\mathbf{j}-2\mathbf{k}&\mathbf{i}-2\mathbf{j}+6\mathbf{k}&1+\mathbf{i}-2\mathbf{j}+4\mathbf{k}\\
\mathbf{i}-4\mathbf{j}+3\mathbf{k}&3-2\mathbf{i}+3\mathbf{j}-6\mathbf{k}&1-\mathbf{i}-2\mathbf{j}+\mathbf{k}\\
4\mathbf{i}+\mathbf{j}-6\mathbf{k}&1-2\mathbf{i}-4\mathbf{j}+8\mathbf{k}&2-7\mathbf{j}+7\mathbf{k}\end{pmatrix},
\end{align*}
\begin{align*}
C_{2}=\begin{pmatrix}3-3\mathbf{i}+3\mathbf{j}+3\mathbf{k}&-2-\mathbf{i}+\mathbf{j}-9\mathbf{k}&-6-6\mathbf{i}-3\mathbf{j}+8\mathbf{k}\\
-2-8\mathbf{j}+8\mathbf{k}&\mathbf{i}+2\mathbf{j}&-4-6\mathbf{i}-5\mathbf{j}-6\mathbf{k}\\
1-3\mathbf{i}-5\mathbf{j}+11\mathbf{k}&-2+3\mathbf{j}-9\mathbf{k}&-10-12\mathbf{i}-8\mathbf{j}+2\mathbf{k}\end{pmatrix},
\end{align*}
\begin{align*}
C_{3}=\begin{pmatrix}3-4\mathbf{i}+4\mathbf{j}+3\mathbf{k}&-2+6\mathbf{i}+2\mathbf{j}-\mathbf{k}&3-\mathbf{i}-\mathbf{j}\\
4-\mathbf{i}+6\mathbf{j}-\mathbf{k}&-4-\mathbf{i}-4\mathbf{j}+10\mathbf{k}&-2-2\mathbf{i}+6\mathbf{j}-3\mathbf{k}\\
1-4\mathbf{i}+8\mathbf{j}+4\mathbf{k}&-1+8\mathbf{i}+3\mathbf{k}&3-\mathbf{i}+3\mathbf{j}-\mathbf{k}\end{pmatrix},
\end{align*}
\begin{align*}
C_{4}=\begin{pmatrix}2+\mathbf{i}-3\mathbf{k}&2+2\mathbf{j}+3\mathbf{k}&8-2\mathbf{k}\\
-1+2\mathbf{i}-3\mathbf{j}+\mathbf{k}&-3+6\mathbf{i}-4\mathbf{j}+\mathbf{k}&-1-3\mathbf{i}+3\mathbf{j}\\
1-3\mathbf{i}+5\mathbf{j}-3\mathbf{k}&1-3\mathbf{j}+2\mathbf{k}&4-2\mathbf{i}+\mathbf{j}+3\mathbf{k}\end{pmatrix}.
\end{align*}
Now we consider the system of one-sided coupled Sylvester-type real quaternion matrix equations (\ref{sys02}). Check that
\begin{align*}
r\begin{pmatrix}C_{i}&A_{i}\\B_{i}&0\end{pmatrix}=r(A_{i})+r(B_{i})=
\left\{\begin{array}{lll}
5, &\mbox{if}~
 i=1\\
3, & \mbox{if}~
 i=2\\
4, &\mbox{if}~
 i=3\\
5, & \mbox{if}~
 i=4
\end{array},
 \right.
\end{align*}
\begin{align*}
r\begin{pmatrix}A_{2}C_{1}+C_{2}B_{1}&A_{2}A_{1}\\B_{2}B_{1}&0\end{pmatrix}=r(A_{2}A_{1})+r(B_{2}B_{1})=3,
\end{align*}
\begin{align*}
r\begin{pmatrix}C_{3}&C_{4}&A_{3}&A_{4}\\B_{3}&B_{4}&0&0\end{pmatrix}=r(A_{3},~A_{4})+r(B_{3},~B_{4})=6,
\end{align*}
\begin{align*}
r\begin{pmatrix}A_{3}C_{2}+C_{3}B_{2}&C_{4}&A_{3}A_{2}&A_{4}\\B_{3}B_{2}&B_{4}&0&0\end{pmatrix}
=r(A_{3}A_{2},~A_{4})+r(B_{3}B_{2},~B_{4})=6,
\end{align*}
\begin{align*}
r\begin{pmatrix}A_{3}A_{2}C_{1}+A_{3}C_{2}B_{1}+C_{3}B_{2}B_{1}&A_{3}A_{2}A_{1}\\B_{3}B_{2}B_{1}&0\end{pmatrix}
=r(A_{3}A_{2}A_{1})
+r(B_{3}B_{2}B_{1})=3,
\end{align*}
\begin{align*}
r\begin{pmatrix}A_{3}C_{2}+C_{3}B_{2}&A_{3}A_{2}\\B_{3}B_{2}&0\end{pmatrix}=r(A_{3}A_{2})+r(B_{3}B_{2})=3,
\end{align*}
\begin{align*}
r\begin{pmatrix}A_{3}A_{2}C_{1}+A_{3}C_{2}B_{1}+C_{3}B_{2}B_{1}&C_{4}&A_{4}&A_{3}A_{2}A_{1}\\B_{3}B_{2}B_{1}&B_{4}&0&0\end{pmatrix}
=r(A_{3}A_{2}A_{1},~A_{4})+r(B_{3}B_{2}B_{1},~B_{4})=6.
\end{align*}
All the rank equalities in (\ref{equhh047})-(\ref{equhh413}) hold. Hence, the system of one-sided coupled Sylvester-type real quaternion matrix equations (\ref{sys02}) is consistent. Note that
\begin{align*}
X_{1}=\begin{pmatrix}1-\mathbf{k}&\mathbf{i}+\mathbf{j}+2\mathbf{k}&\mathbf{k}\\
-1+\mathbf{j}&-\mathbf{i}-2\mathbf{j}+\mathbf{k}&-\mathbf{j}\\
\mathbf{j}-\mathbf{k}&-\mathbf{j}+3\mathbf{k}&-\mathbf{j}+\mathbf{k}\end{pmatrix}
~
X_{2}=\begin{pmatrix}1+\mathbf{j}&\mathbf{i}&-1+\mathbf{k}\\
\mathbf{i}+\mathbf{j}+\mathbf{k}&-1&-\mathbf{i}-\mathbf{j}+\mathbf{k}\\
1+\mathbf{i}+2\mathbf{j}+\mathbf{k}&-1+\mathbf{i}&-1-\mathbf{i}-\mathbf{j}+2\mathbf{k}\end{pmatrix},
\end{align*}
\begin{align*}
X_{3}=\begin{pmatrix}-2\mathbf{k}&2+\mathbf{k}&\mathbf{i}\\
\mathbf{i}+2\mathbf{j}&1-\mathbf{j}&1-\mathbf{i}\\
\mathbf{i}+2\mathbf{j}-2\mathbf{k}&3-\mathbf{j}+\mathbf{k}&1\end{pmatrix}
~
X_{4}=\begin{pmatrix}\mathbf{k}&\mathbf{i}&1-\mathbf{k}\\
1-2\mathbf{i}+\mathbf{j}-\mathbf{k}&1-3\mathbf{i}&1+\mathbf{i}+2\mathbf{j}+\mathbf{k}\\
-1&2+\mathbf{k}&\mathbf{i}\end{pmatrix},
\end{align*}and
\begin{align*}
X_{5}=\begin{pmatrix}\mathbf{i}+\mathbf{j}&\mathbf{k}&1+\mathbf{k}\\
1+2\mathbf{j}&\mathbf{i}&1+\mathbf{j}\\
\mathbf{j}+\mathbf{k}&\mathbf{k}&1+2\mathbf{j}\end{pmatrix}
\end{align*}
satisfy the system (\ref{sys02}).

\end{example}

Let $A_{4},B_{4},$ and $C_{4}$ vanish in Theorem \ref{theorem02}. Then we can obtain some necessary and sufficient conditions and general solution to the system of coupled generalized  Sylvester real quaternion matrix
equations (\ref{specialsys02}).

\begin{corollary}\cite{auto001}
Let $A_{i},B_{i},$ and $C_{i}(i=1,2,3)$ be given. Set
\begin{align*}
&A=R_{(A_{2}A_{1})}A_{2},B=R_{B_{1}}L_{(B_{3}B_{2})},C=R_{(A_{2}A_{1})}L_{A_{3}},\\&D=B_{2}L_{(B_{3}B_{2})},M=R_{A}C,N=DL_{B},S=CL_{M},\\&
C_{4}=C_{2}+A_{3}^{\dag}C_{3}B_{2}+A_{2}R_{A_{1}}C_{1}B_{1}^{\dag},E=R_{(A_{2}A_{1})}C_{4}L_{(B_{3}B_{2})}.
\end{align*}
Then the following statements are equivalent:\\
$(1)$ The system of coupled generalized  Sylvester real quaternion matrix
equations (\ref{specialsys02}) is consistent.\\
$(2)$
\begin{align*}
&r\begin{pmatrix}C_{i}&A_{i}\\B_{i}&0\end{pmatrix}=r(A_{i})+r(B_{i}),(i=1,2,3),
\\&
r\begin{pmatrix}A_{3}C_{2}+C_{3}B_{2}&A_{3}A_{2}\\B_{3}B_{2}&0\end{pmatrix}=r(A_{3}A_{2})
+r(B_{3}B_{2}), \\&
r\begin{pmatrix}A_{2}C_{1}+C_{2}B_{1}&A_{2}A_{1}\\B_{2}B_{1}&0\end{pmatrix}=r(A_{2}A_{1})
+r(B_{2}B_{1}),
\\
&r\begin{pmatrix}A_{3}A_{2}C_{1}+A_{3}C_{2}B_{1}+C_{3}B_{2}B_{1}&A_{3}A_{2}A_{1}\\B_{3}B_{2}B_{1}&0\end{pmatrix}\\&=r(A_{3}A_{2}A_{1})
+r(B_{3}B_{2}B_{1}).
\end{align*}

$(3)$
\begin{align*}
&R_{A_{i}}C_{i}L_{B_{i}}=0,(i=1,2),R_{M}R_{A}E=0,\\&
EL_{B}L_{N}=0, R_{A}EL_{D}=0,  R_{C}EL_{B}=0.
\end{align*}

In this case, the general solution to the coupled generalized  Sylvester real quaternion matrix
equations (\ref{specialsys02}) can be expressed as
\begin{align*}
X=A_{1}^{\dag}C_{1}+U_{1}B_{1}+L_{A_{1}}U_{2},~
Y=-R_{A_{1}}C_{1}B_{1}^{\dag}+A_{1}U_{1}+U_{3}R_{B_{1}},
\end{align*}
\begin{align*}
Z=A_{3}^{\dag}C_{3}-V_{1}B_{3}-L_{A_{3}}V_{2},~
W=-R_{A_{3}}C_{3}B_{3}^{\dag}-A_{3}V_{1}-V_{3}R_{B_{3}},
\end{align*}
where
\begin{align*}
  U_{1}=(A_{2}A_{1})^{\dag}(C_{4}-A_{2}U_{3}R_{B_{1}}-L_{A_{3}}V_{2}B_{2})-(A_{2}A_{1})^{\dag}%
T_{7}(B_{3}B_{2})+L_{(A_{2}A_{1})}T_{6},
\end{align*}
\begin{align*}
  V_{1}=R_{(A_{2}A_{1})}(C_{4}-A_{2}U_{3}R_{B_{1}}-L_{A_{3}}V_{2}B_{2})(B_{3}B_{2})^{\dag}%
+(A_{2}A_{1})(A_{2}A_{1})^{\dag}T_{7}+T_{8}R_{(B_{3}B_{2})},
\end{align*}
\begin{align*}
U_{3}=A^{\dag}EB^{\dag}-A^{\dag}CM^{\dag}EB^{\dag}-A^{\dag}SC^{\dag
}EN^{\dag}DB^{\dag} -A^{\dag}ST_{2}R_{N}DB^{\dag}+L_{A}T_{4}+T_{5}%
R_{B},
\end{align*}
\begin{align*}
 V_{2}=M^{\dag}ED^{\dag}+S^{\dag}SC^{\dag}EN^{\dag}+L_{M}L_{S}T_{1}%
+L_{M}T_{2}R_{N}+T_{3}R_{D},
\end{align*}
and $U_{2},V_{3},T_{1}, \ldots, T_{8}$ are arbitrary matrices over
$\mathbb{H}$ with appropriate sizes.
\end{corollary}

\section{\textbf{Some solvability conditions and the general solution to
system (\ref{sys03})}}

In this section, we consider the solvability conditions and the general solution to the
system of one-sided coupled Sylvester-type real quaternion matrix equations (\ref{sys03}). For simplicity, put
\begin{align*}
A_{11}=R_{(A_{2}A_{1})}A_{2},~B_{11}=R_{B_{1}}L_{B_{2}},~C_{11}=R_{(A_{2}A_{1})}(A_{2}R_{A_{1}}C_{1}B_{1}^{\dag}+C_{2})L_{B_{2}},
\end{align*}
\begin{align*}
A_{22}=R_{(A_{3}A_{4})}A_{3},~B_{22}=R_{B_{4}}L_{B_{3}},~C_{22}=R_{(A_{3}A_{4})}(A_{3}R_{A_{4}}C_{4}B_{4}^{\dag}+C_{3})L_{B_{3}},
\end{align*}
\begin{align*}
A_{33}=(A_{2}A_{1},~-A_{3}A_{4}),~B_{33}=\begin{pmatrix}R_{B_{2}}\\-R_{B_{3}}\end{pmatrix},
~A_{44}=R_{B_{11}}R_{B_{1}}B_{2}^{\dag},~B_{44}=R_{B_{22}}R_{B_{4}}B_{3}^{\dag},
\end{align*}
\begin{align*}
E_{1}=&R_{(A_{2}A_{1})}C_{2}B_{2}^{\dag}+A_{11}R_{A_{1}}C_{1}B_{1}^{\dag}B_{2}^{\dag}-C_{11}B_{11}^{\dag}R_{B_{1}}B_{2}^{\dag}-\\&
R_{(A_{3}A_{4})}C_{3}B_{3}^{\dag}-A_{22}R_{A_{4}}C_{4}B_{4}^{\dag}B_{2}^{\dag}+C_{22}B_{22}^{\dag}R_{B_{4}}B_{3}^{\dag},
\end{align*}
\begin{align*}
A=R_{A_{33}}A_{11},~B=A_{44}L_{B_{33}},~C=-R_{A_{33}}A_{22},~D=B_{44}L_{B_{33}},
\end{align*}
\begin{align*}
E=R_{A_{33}}E_{1}L_{B_{33}},M=R_{A}C,N=DL_{B},S=CL_{M}.
\end{align*}

Now we give the fundamental theorem of this section.
\begin{theorem}\label{theorem03}
Let $A_{i},B_{i},$ and $C_{i}(i=1,2,3,4)$ be given. Then the following statements are equivalent:\\$(1)$ The system of one-sided coupled Sylvester-type real quaternion matrix equations (\ref{sys03}) is consistent.\\$(2)$
\begin{align}\label{equhh056}
r\begin{pmatrix}C_{i}&A_{i}\\B_{i}&0\end{pmatrix}=r(A_{i})+r(B_{i}),\quad (i=1,2,3,4),
\end{align}
\begin{align}
r\begin{pmatrix}A_{2}C_{1}+C_{2}B_{1}&A_{2}A_{1}\\B_{2}B_{1}&0\end{pmatrix}=r(A_{2}A_{1})+r(B_{2}B_{1}),
\end{align}
\begin{align}
r\begin{pmatrix}A_{3}C_{4}+C_{3}B_{4}&A_{3}A_{4}\\B_{3}B_{4}&0\end{pmatrix}=r(A_{3}A_{4})+r(B_{3}B_{4}),
\end{align}
\begin{align}
r\begin{pmatrix}C_{2}&C_{3}&A_{2}&A_{3}\\B_{2}&B_{3}&0&0\end{pmatrix}=r(A_{2},~A_{3})+r(B_{2},~B_{3}),
\end{align}
\begin{align}
r\begin{pmatrix}A_{2}C_{1}+C_{2}B_{1}&A_{3}C_{4}+C_{3}B_{4}&A_{2}A_{1}&A_{3}A_{4}\\B_{2}B_{1}&B_{3}B_{4}&0&0\end{pmatrix}
=r(A_{2}A_{1},~A_{3}A_{4})+r(B_{2}B_{1},~B_{3}B_{4}),
\end{align}
\begin{align}
r\begin{pmatrix}C_{2}&A_{3}C_{4}+C_{3}B_{4}&A_{2}&A_{3}A_{4}\\B_{2}&B_{3}B_{4}&0&0\end{pmatrix}
=r(A_{2},~A_{3}A_{4})+r(B_{2},~B_{3}B_{4}),
\end{align}
\begin{align}\label{equhh512}
r\begin{pmatrix}A_{2}C_{1}+C_{2}B_{1}&C_{3}&A_{2}A_{1}&A_{3}\\B_{2}B_{1}&B_{3}&0&0\end{pmatrix}
=r(A_{2}A_{1},~A_{3})+r(B_{2}B_{1},~B_{3}).
\end{align}

$(3)$
\begin{align*}
R_{A_{1}}C_{1}L_{B_{1}}=0,~R_{A_{11}}C_{11}=0,~C_{11}L_{B_{11}}=0,
\end{align*}
\begin{align*}
R_{A_{4}}C_{4}L_{B_{4}}=0,~R_{A_{22}}C_{22}=0,~C_{22}L_{B_{22}}=0,
\end{align*}
\begin{align*}
R_{M}R_{A}E=0,~EL_{B}L_{N}=0,~R_{A}EL_{D}=0,~R_{C}EL_{B}=0.
\end{align*}

In this case, the general solution to the system of one-sided coupled Sylvester-type real quaternion matrix equations (\ref{sys03}) can be expressed as
\begin{align*}
X_{1}=A_{1}^{\dag}C_{1}+U_{1}B_{1}+L_{A_{1}}W_{1},~
X_{2}=-R_{A_{1}}C_{1}B_{1}^{\dag}+A_{1}U_{1}+V_{1}R_{B_{1}},
\end{align*}
\begin{align*}
X_{4}=-R_{A_{4}}C_{4}B_{4}^{\dag}+A_{4}U_{2}+V_{2}R_{B_{4}},
~
X_{5}=A_{4}^{\dag}C_{4}+U_{2}B_{4}+L_{A_{4}}T_{1},
\end{align*}
\begin{align*}
X_{3}=-R_{(A_{2}A_{1})}(C_{2}+A_{2}R_{A_{1}}C_{1}B_{1}^{\dag}-A_{2}V_{1}R_{B_{1}})B_{2}^{\dag}+A_{2}A_{1}W_{4}+W_{5}R_{B_{2}},
\end{align*}
or
\begin{align*}
X_{3}=-R_{(A_{3}A_{4})}(C_{3}+A_{3}R_{A_{4}}C_{4}B_{4}^{\dag}-A_{3}V_{2}R_{B_{4}})B_{3}^{\dag}+A_{3}A_{4}T_{4}+T_{5}R_{B_{3}},
\end{align*}
where
\begin{align*}
V_{1}=A_{11}^{\dag}C_{11}B_{11}^{\dag}+L_{A_{11}}W_{2}+W_{3}R_{B_{11}},
\end{align*}
\begin{align*}
U_{1}=(A_{2}A_{1})^{\dag}(C_{2}+A_{2}R_{A_{1}}C_{1}B_{1}^{\dag}-A_{2}V_{1}R_{B_{1}})+W_{4}B_{2}+L_{(A_{2}A_{1})}W_{6},
\end{align*}
\begin{align*}
V_{2}=A_{22}^{\dag}C_{22}B_{22}^{\dag}+L_{A_{22}}T_{2}+T_{3}R_{B_{22}},
\end{align*}
\begin{align*}
U_{2}=(A_{3}A_{4})^{\dag}(C_{3}+A_{3}R_{A_{4}}C_{4}B_{4}^{\dag}-A_{3}V_{2}R_{B_{4}})+T_{4}B_{3}+L_{(A_{3}A_{4})}T_{6},
\end{align*}
\begin{align*}
W_{4}=(I_{p_{1}},~0)[
A_{33}^{\dag}(E_{1}-A_{11}W_{3}A_{44}+A_{22}T_{3}B_{44})-A_{33}^{\dag}Z_{7}B_{33}+L_{A_{33}}Z_{6}],
\end{align*}
\begin{align*}
T_{4}=(0,~I_{p_{2}})[
A_{33}^{\dag}(E_{1}-A_{11}W_{3}A_{44}+A_{22}T_{3}B_{44})-A_{33}^{\dag}Z_{7}B_{33}+L_{A_{33}}Z_{6}],
\end{align*}
\begin{align*}
W_{5}=[R_{A_{33}}(E_{1}-A_{11}W_{3}A_{44}+A_{22}T_{3}B_{44})B_{33}^{\dag}
+A_{33}A_{33}^{\dag}Z_{7}+Z_{8}R_{B_{33}}]\begin{pmatrix}I_{p_{3}}\\0\end{pmatrix},
\end{align*}
\begin{align*}
T_{5}=[R_{A_{33}}(E_{1}-A_{11}W_{3}A_{44}+A_{22}T_{3}B_{44})B_{33}^{\dag}
+A_{33}A_{33}^{\dag}Z_{7}+Z_{8}R_{B_{33}}]\begin{pmatrix}0\\I_{p_{4}}\end{pmatrix},
\end{align*}
\begin{align*}
W_{3}=
A^{\dag}EB^{\dag}-A^{\dag}CM^{\dag}EB^{\dag}-A^{\dag}SC^{\dag
}EN^{\dag}DB^{\dag}-A^{\dag}SZ_{1}R_{N}DB^{\dag}+L_{A}Z_{2}+Z_{3}%
R_{B},
\end{align*}
\begin{align*}
T_{3}=M^{\dag}ED^{\dag}+S^{\dag}SC^{\dag}EN^{\dag}+L_{M}L_{S}Z_{4}%
+L_{M}Z_{1}R_{N}+Z_{5}R_{D},
\end{align*}
the remaining $W_{j},T_{j},Z_{j}$ are  arbitrary matrices over $\mathbb{H}$, $p_{1}$ and $p_{2}$ are the column numbers of $A_{1}$ and $A_{4}$, respectively, $p_{3}$ and $p_{4}$ are the row numbers of $B_{2}$ and $B_{3}$, respectively.

\end{theorem}

\begin{proof}
We separate this system of one-sided coupled Sylvester-type real quaternion matrix equations (\ref{sys03}) into two parts
\begin{align}
  \left\{\begin{array}{c}
A_{1}X_{1}-X_{2}B_{1}=C_{1},\\
A_{2}X_{2}-X_{3}B_{2}=C_{2},
\end{array}
  \right.
\end{align}
and
\begin{align}
  \left\{\begin{array}{c}
A_{3}X_{4}-X_{3}B_{3}=C_{3},\\
A_{4}X_{5}-X_{4}B_{4}=C_{4}.
\end{array}
  \right.
\end{align}
Applying the main idea of Theorem \ref{theorem01}, Lemma \ref{lemma01}, Lemma \ref{lemma02}, Lemma \ref{lemma03} and Lemma \ref{lemma04}, we can prove Theorem \ref{theorem03}.
\end{proof}

Now we give an example to illustrate Theorem \ref{theorem03}.
\begin{example}
Given the quaternion matrices:
\begin{align*}
A_{1}=\begin{pmatrix}\mathbf{i}+\mathbf{j}+\mathbf{k}&2+\mathbf{i}+\mathbf{j}-\mathbf{k}\\
-1+\mathbf{j}+\mathbf{k}&-1+2\mathbf{i}+\mathbf{j}-\mathbf{k}\end{pmatrix}, ~
B_{1}=\begin{pmatrix}1+\mathbf{k}&\mathbf{j}-\mathbf{k}\\
\mathbf{i}+2\mathbf{k}&2\mathbf{j}-2\mathbf{k}\end{pmatrix},
\end{align*}
\begin{align*}
A_{2}=\begin{pmatrix}\mathbf{i}&2+\mathbf{j}\\
1+\mathbf{i}+\mathbf{k}&-\mathbf{j}\end{pmatrix},~
B_{2}=\begin{pmatrix}1+\mathbf{j}+2\mathbf{k}&\mathbf{i}+3\mathbf{k}\\
\mathbf{j}&1+\mathbf{i}\end{pmatrix},
\end{align*}
\begin{align*}
A_{3}=\begin{pmatrix}1+\mathbf{k}&\mathbf{i}+\mathbf{k}\\
1+\mathbf{i}+\mathbf{j}+\mathbf{k}&-1+\mathbf{i}+\mathbf{j}+\mathbf{k}\end{pmatrix},~
B_{3}=\begin{pmatrix}2+\mathbf{k}&-\mathbf{i}-2\mathbf{j}\\
2-2\mathbf{i}-\mathbf{k}&-1-\mathbf{i}+2\mathbf{j}\end{pmatrix},
\end{align*}
\begin{align*}
A_{4}=\begin{pmatrix}3\mathbf{i}+\mathbf{j}&1+2\mathbf{j}\\
2+\mathbf{k}&0\end{pmatrix},~
B_{4}=\begin{pmatrix}-\mathbf{j}&1+2\mathbf{j}\\
-\mathbf{k}&\mathbf{i}+2\mathbf{k}\end{pmatrix},
\end{align*}
\begin{align*}
C_{1}=\begin{pmatrix}1-5\mathbf{i}+\mathbf{j}+\mathbf{k}&3+8\mathbf{i}-7\mathbf{j}-3\mathbf{k}\\
3+2\mathbf{i}-2\mathbf{j}-7\mathbf{k}&-8+3\mathbf{i}-4\mathbf{j}-4\mathbf{k}\end{pmatrix},
~
C_{2}=\begin{pmatrix}-1+\mathbf{i}+3\mathbf{j}+2\mathbf{k}&5+3\mathbf{i}+3\mathbf{j}+7\mathbf{k}\\
1+2\mathbf{i}+4\mathbf{j}+5\mathbf{k}&-3\mathbf{i}-2\mathbf{j}+5\mathbf{k}\end{pmatrix},
\end{align*}
\begin{align*}
C_{3}=\begin{pmatrix}-7\mathbf{i}-2\mathbf{j}-2\mathbf{k}&-3-7\mathbf{i}+4\mathbf{j}+5\mathbf{k}\\
-3+\mathbf{i}-3\mathbf{j}-\mathbf{k}&-6-3\mathbf{i}-\mathbf{j}+5\mathbf{k}\end{pmatrix},
~
C_{4}=\begin{pmatrix}-2+7\mathbf{i}-4\mathbf{j}&2-7\mathbf{i}+4\mathbf{j}+8\mathbf{k}\\
-6+\mathbf{i}-\mathbf{j}+\mathbf{k}&13-3\mathbf{i}+\mathbf{j}+4\mathbf{k}\end{pmatrix}.
\end{align*}
Now we consider the system of one-sided coupled Sylvester-type real quaternion matrix equations (\ref{sys03}). Check that
\begin{align*}
r\begin{pmatrix}C_{i}&A_{i}\\B_{i}&0\end{pmatrix}=r(A_{i})+r(B_{i})=
\left\{\begin{array}{lll}
4, &\mbox{if}~
 i=1,2,3\\
3, & \mbox{if}~
 i=4
\end{array},
 \right.
\end{align*}
\begin{align*}
r\begin{pmatrix}A_{2}C_{1}+C_{2}B_{1}&A_{2}A_{1}\\B_{2}B_{1}&0\end{pmatrix}=r(A_{2}A_{1})+r(B_{2}B_{1})=4,
\end{align*}
\begin{align*}
r\begin{pmatrix}A_{3}C_{4}+C_{3}B_{4}&A_{3}A_{4}\\B_{3}B_{4}&0\end{pmatrix}=r(A_{3}A_{4})+r(B_{3}B_{4})=3,
\end{align*}
\begin{align*}
r\begin{pmatrix}C_{2}&C_{3}&A_{2}&A_{3}\\B_{2}&B_{3}&0&0\end{pmatrix}=r(A_{2},~A_{3})+r(B_{2},~B_{3})=4,
\end{align*}
\begin{align*}
r\begin{pmatrix}A_{2}C_{1}+C_{2}B_{1}&A_{3}C_{4}+C_{3}B_{4}&A_{2}A_{1}&A_{3}A_{4}\\B_{2}B_{1}&B_{3}B_{4}&0&0\end{pmatrix}
=r(A_{2}A_{1},~A_{3}A_{4})+r(B_{2}B_{1},~B_{3}B_{4})=4,
\end{align*}
\begin{align*}
r\begin{pmatrix}C_{2}&A_{3}C_{4}+C_{3}B_{4}&A_{2}&A_{3}A_{4}\\B_{2}&B_{3}B_{4}&0&0\end{pmatrix}
=r(A_{2},~A_{3}A_{4})+r(B_{2},~B_{3}B_{4})=4,
\end{align*}
\begin{align*}
r\begin{pmatrix}A_{2}C_{1}+C_{2}B_{1}&C_{3}&A_{2}A_{1}&A_{3}\\B_{2}B_{1}&B_{3}&0&0\end{pmatrix}
=r(A_{2}A_{1},~A_{3})+r(B_{2}B_{1},~B_{3})=4.
\end{align*}
All the rank equalities in (\ref{equhh056})-(\ref{equhh512}) hold. Hence, the system of one-sided coupled Sylvester-type real quaternion matrix equations (\ref{sys03}) is consistent. Note that
\begin{align*}
X_{1}=\begin{pmatrix}\mathbf{i}+\mathbf{j}&-1+\mathbf{k}\\
2+\mathbf{k}&2\mathbf{i}-\mathbf{j}\end{pmatrix}
~
X_{2}=\begin{pmatrix}1+2\mathbf{i}+\mathbf{j}&-\mathbf{i}+2\mathbf{j}\\
\mathbf{k}&1+2\mathbf{k}\end{pmatrix},
\end{align*}
\begin{align*}
X_{3}=\begin{pmatrix}\mathbf{i}&-1+\mathbf{j}\\
-1&-\mathbf{i}+\mathbf{j}+\mathbf{k}\end{pmatrix}
~
X_{4}=\begin{pmatrix}-1+2\mathbf{j}&1+3\mathbf{j}\\
-\mathbf{i}+2\mathbf{j}+2\mathbf{k}&\mathbf{i}+3\mathbf{j}+3\mathbf{k}\end{pmatrix},
\end{align*}and
\begin{align*}
X_{5}=\begin{pmatrix}\mathbf{k}&1+2\mathbf{j}\\ \mathbf{i}+\mathbf{k}&1-\mathbf{i}+\mathbf{j}-\mathbf{k}\end{pmatrix}
\end{align*}
satisfy the system (\ref{sys03}).

\end{example}

Let $A_{4},B_{4},$ and $C_{4}$ vanish in Theorem \ref{theorem03}. Then we can obtain some necessary and sufficient conditions and general solution to the system of coupled generalized  Sylvester real quaternion matrix
equations (\ref{specialsys03}).

\begin{corollary}\cite{auto001}
Let $A_{i},B_{i},$ and $C_{i}(i=1,2,3)$ be given. Set
\begin{align*}
&A=R_{(A_{2}A_{1})}A_{2},B=R_{B_{1}}L_{(R_{B_{3}}B_{2})},C=R_{(A_{2}A_{1})}A_{3},\\&D=B_{2}L_{(R_{B_{3}}B_{2})},
C_{4}=C_{2}+A_{2}^{\dag}R_{A_{1}}C_{1}B_{1}^{\dag}-
R_{A_{3}}C_{3}B_{3}^{\dag}B_{2},\\&E=R_{(A_{2}A_{1})}C_{4}L_{(R_{B_{3}}B_{2})},M=R_{A}C,N=DL_{B},S=CL_{M}.
\end{align*}
Then the following statements are equivalent:\\
$(1)$ The system of coupled generalized  Sylvester real quaternion matrix
equations (\ref{specialsys03}) is consistent.\\
$(2)$
\begin{align*}
r\begin{pmatrix}C_{i}&A_{i}\\B_{i}&0\end{pmatrix}=r(A_{i})+r(B_{i}),(i=1,2,3),
\end{align*}
\begin{align*}
r\begin{pmatrix}A_{2}&A_{3}&C_{2}&C_{3}\\0&0&B_{2}&B_{3}\end{pmatrix}=r(A_{2},~A_{3})
+r(B_{2},~B_{3}),
\end{align*}
\begin{align*}
r\begin{pmatrix}A_{2}C_{1}+C_{2}B_{1}&A_{2}A_{1}\\B_{2}B_{1}&0\end{pmatrix}=r(A_{2}A_{1})
+r(B_{2}B_{1}),
\end{align*}
\begin{align*}
r\begin{pmatrix}A_{3}&A_{2}A_{1}&C_{3}&A_{2}C_{1}+C_{2}B_{1}\\0&0&B_{3}&
B_{2}B_{1} \end{pmatrix}=r(A_{3},~A_{2}A_{1}) +r(B_{3},~B_{2}B_{1}).
\end{align*}
$(3)$
\begin{align*}
&R_{A_{i}}C_{i}L_{B_{i}}=0,(i=1,2),R_{M}R_{A}E=0,\\&
EL_{B}L_{N}=0, R_{A}EL_{D}=0,  R_{C}EL_{B}=0.
\end{align*}

In this case, the general solution to the coupled generalized  Sylvester real quaternion matrix
equations (\ref{specialsys03}) can be expressed as
\begin{align*}
X=A_{1}^{\dag}C_{1}+U_{1}B_{1}+L_{A_{1}}U_{2},~
Y=-R_{A_{1}}C_{1}B_{1}^{\dag}+A_{1}U_{1}+U_{3}R_{B_{1}},
\end{align*}
\begin{align*}
Z=-R_{A_{3}}C_{3}B_{3}^{\dag}-A_{3}V_{1}-V_{3}R_{B_{3}},~
W=A_{3}^{\dag}C_{3}-V_{1}B_{3}-L_{A_{3}}V_{2},
\end{align*}
where
\begin{align*}
  U_{1}=(A_{2}A_{1})^{\dag}(C_{4}-A_{2}U_{3}R_{B_{1}}-A_{3}V_{1}B_{2})-(A_{2}A_{1})^{\dag}%
T_{7}(R_{B_{3}}B_{2})+L_{(A_{2}A_{1})}T_{6},
\end{align*}
\begin{align*}
  V_{3}=R_{(A_{2}A_{1})}(C_{4}-A_{2}U_{3}R_{B_{1}}-A_{3}V_{1}B_{2})(R_{B_{3}}B_{2})^{\dag}%
+(A_{2}A_{1})(A_{2}A_{1})^{\dag}T_{7}+T_{8}R_{(R_{B_{3}}B_{2})},
\end{align*}
\begin{align*}
U_{3}=A^{\dag}EB^{\dag}-A^{\dag}CM^{\dag}EB^{\dag}-A^{\dag}SC^{\dag
}EN^{\dag}DB^{\dag} -A^{\dag}ST_{2}R_{N}DB^{\dag}+L_{A}T_{4}+T_{5}%
R_{B},
\end{align*}
\begin{align*}
 V_{1}=M^{\dag}ED^{\dag}+S^{\dag}SC^{\dag}EN^{\dag}+L_{M}L_{S}T_{1}%
+L_{M}T_{2}R_{N}+T_{3}R_{D},
\end{align*}
and $U_{2},V_{2},T_{1}, \ldots, T_{8}$ are arbitrary matrices over
$\mathbb{H}$ with appropriate sizes.
\end{corollary}

\section{\textbf{Some solvability conditions and the general solution to
system (\ref{sys04})}}

Our goal of this section is to give some necessary and
sufficient conditions and the general solution
 to the system (\ref{sys04}). Set
\begin{align*}
A_{jj}=R_{(A_{2j}A_{2j-1})}A_{2j},B_{jj}=R_{B_{2j-1}}L_{B_{2j}},
C_{jj}=R_{(A_{2j}A_{2j-1})}(A_{2j}R_{A_{2j-1}}C_{2j-1}B_{2j-1}^{\dag}+C_{2j})L_{B_{2j}},
\end{align*}
\begin{align*}
(j=1,2),~A_{33}=(A_{2}A_{1},-L_{A_{3}}),~B_{33}=\begin{pmatrix}R_{B_{2}}\\-B_{4}B_{3}\end{pmatrix},~A_{44}=R_{B_{11}}R_{B_{1}}B_{2}^{\dag},B_{44}=-L_{(A_{4}A_{3})},
\end{align*}
\begin{align*}
E_{1}=A_{3}^{\dag}C_{3}+(A_{4}A_{3})^{\dag}C_{4}B_{3}+(A_{4}A_{3})^{\dag}A_{4}R_{A_{3}}C_{3}
+R_{(A_{2}A_{1})}C_{2}B_{2}^{\dag}+A_{11}R_{A_{1}}C_{1}B_{1}^{\dag}B_{2}^{\dag}-C_{11}B_{11}^{\dag}R_{B_{1}}B_{2}^{\dag},
\end{align*}
\begin{align*}
A=R_{A_{33}}A_{11},~B=A_{44}L_{B_{33}},~C=R_{A_{33}}B_{44},~D=B_{3}L_{B_{33}},
\end{align*}
\begin{align*}
M=R_{A}C,~N=DL_{B},~S=CL_{M},~E=R_{A_{33}}E_{1}L_{B_{33}}.
\end{align*}

Now we give the fundamental theorem of this section.

\begin{theorem}\label{theorem04}
Let $A_{i},B_{i},$ and $C_{i}(i=1,2,3,4)$ be given. Then the following statements are equivalent:\\
$(1)$ The system of one-sided coupled Sylvester-type real quaternion matrix equations (\ref{sys04}) is consistent.\\
$(2)$
\begin{align}\label{equhh061}
r\begin{pmatrix}C_{i}&A_{i}\\B_{i}&0\end{pmatrix}=r(A_{i})+r(B_{i}),\quad (i=1,2,3,4),
\end{align}
\begin{align}
r\begin{pmatrix}A_{k+1}C_{k}+C_{k+1}B_{k}&A_{k+1}A_{k}\\B_{k+1}B_{k}&0\end{pmatrix}=r(A_{k+1}A_{k})
+r(B_{k+1}B_{k}),~(k=1,2,3),
\end{align}
\begin{align}
&r\begin{pmatrix}A_{j+2}A_{j+1}C_{j}+A_{j+2}C_{j+1}B_{j}+C_{j+2}B_{j+1}B_{j}&A_{j+2}A_{j+1}A_{j}\\B_{j+2}B_{j+1}B_{j}&0\end{pmatrix}\nonumber\\&=r(A_{j+2}A_{j+1}A_{j})
+r(B_{j+2}B_{j+1}B_{j}),~(j=1,2),
\end{align}
\begin{align}\label{equhh064}
&r\begin{pmatrix}A_{4}A_{3}A_{2}C_{1}+A_{4}A_{3}C_{2}A_{1}+A_{4}C_{3}A_{2}A_{1}+C_{4}A_{3}A_{2}A_{1}&A_{4}A_{3}A_{2}A_{1}\\B_{4}B_{3}B_{2}B_{1}&0\end{pmatrix}\nonumber\\&
=r(A_{4}A_{3}A_{2}A_{1})
+r(B_{4}B_{3}B_{2}B_{1}).
\end{align}

$(3)$
\begin{align*}
R_{A_{2j-1}}C_{2j-1}L_{B_{2j-1}}=0,~R_{A_{jj}}C_{jj}=0,~~C_{jj}L_{B_{jj}}=0,~(j=1,2),
\end{align*}
\begin{align*}
R_{M}R_{A}E=0,~EL_{B}L_{N}=0,~R_{A}EL_{D}=0,~R_{C}EL_{B}=0.
\end{align*}

In this case, the general solution to the system of one-sided coupled Sylvester-type real quaternion matrix equations (\ref{sys04}) can be expressed as
\begin{align*}
X_{1}=A_{1}^{\dag}C_{1}+U_{1}B_{1}+L_{A_{1}}W_{1},~
X_{2}=-R_{A_{1}}C_{1}B_{1}^{\dag}+A_{1}U_{1}+V_{1}R_{B_{1}},
\end{align*}
\begin{align*}
X_{3}=A_{3}^{\dag}C_{3}+U_{2}B_{3}+L_{A_{3}}T_{1},~
X_{4}=-R_{A_{3}}C_{3}B_{3}^{\dag}+A_{3}U_{2}+V_{2}R_{B_{3}},
\end{align*}
\begin{align*}
X_{3}=-R_{(A_{2}A_{1})}C_{2}B_{2}^{\dag}-A_{11}R_{A_{1}}C_{1}B_{1}^{\dag}B_{2}^{\dag}+C_{11}B_{11}^{\dag}R_{B_{1}}B_{2}^{\dag}+A_{11}W_{3}A_{44}+A_{2}A_{1}W_{4}+W_{5}R_{B_{2}},
\end{align*}
or
\begin{align*}
X_{3}=A_{3}^{\dag}C_{3}+(A_{4}A_{3})^{\dag}C_{4}B_{3}+(A_{4}A_{3})^{\dag}A_{4}R_{A_{3}}C_{3}+L_{A_{3}}T_{1}+T_{4}B_{4}B_{3}+L_{(A_{4}A_{3})}T_{6}B_{3},
\end{align*}
where
\begin{align*}
V_{1}=A_{11}^{\dag}C_{11}B_{11}^{\dag}+L_{A_{11}}W_{2}+W_{3}R_{B_{11}},
\end{align*}
\begin{align*}
U_{1}=(A_{2}A_{1})^{\dag}(C_{2}+A_{2}R_{A_{1}}C_{1}B_{1}^{\dag}-A_{2}V_{1}R_{B_{1}})+W_{4}B_{2}+L_{(A_{2}A_{1})}W_{6},
\end{align*}
\begin{align*}
V_{2}=A_{22}^{\dag}C_{22}B_{22}^{\dag}+L_{A_{22}}T_{2}+T_{3}R_{B_{22}},
\end{align*}
\begin{align*}
U_{2}=(A_{4}A_{3})^{\dag}(C_{4}+A_{4}R_{A_{3}}C_{3}B_{3}^{\dag}-A_{4}V_{2}R_{B_{3}})+T_{4}B_{4}+L_{(A_{4}A_{3})}T_{6},
\end{align*}
\begin{align*}
W_{4}=(I_{p_{1}},~0)[
A_{33}^{\dag}(E_{1}-A_{11}W_{3}A_{44}-B_{44}T_{6}B_{3})-A_{33}^{\dag}Z_{7}B_{33}+L_{A_{33}}Z_{6}],
\end{align*}
\begin{align*}
T_{1}=(0,~I_{p_{2}})[
A_{33}^{\dag}(E_{1}-A_{11}W_{3}A_{44}-B_{44}T_{6}B_{3})-A_{33}^{\dag}Z_{7}B_{33}+L_{A_{33}}Z_{6}],
\end{align*}
\begin{align*}
W_{5}=[R_{A_{33}}(E_{1}-A_{11}W_{3}A_{44}-B_{44}T_{6}B_{3})B_{33}^{\dag}
+A_{33}A_{33}^{\dag}Z_{7}+Z_{8}R_{B_{33}}]\begin{pmatrix}I_{p_{3}}\\0\end{pmatrix},
\end{align*}
\begin{align*}
T_{4}=[R_{A_{33}}(E_{1}-A_{11}W_{3}A_{44}-B_{44}T_{6}B_{3})B_{33}^{\dag}
+A_{33}A_{33}^{\dag}Z_{7}+Z_{8}R_{B_{33}}]\begin{pmatrix}0\\I_{p_{4}}\end{pmatrix},
\end{align*}
\begin{align*}
W_{3}=
A^{\dag}EB^{\dag}-A^{\dag}CM^{\dag}EB^{\dag}-A^{\dag}SC^{\dag
}EN^{\dag}DB^{\dag}-A^{\dag}SZ_{1}R_{N}DB^{\dag}+L_{A}Z_{2}+Z_{3}%
R_{B},
\end{align*}
\begin{align*}
T_{6}=M^{\dag}ED^{\dag}+S^{\dag}SC^{\dag}EN^{\dag}+L_{M}L_{S}Z_{4}%
+L_{M}Z_{1}R_{N}+Z_{5}R_{D},
\end{align*}
the remaining $W_{j},T_{j},Z_{j}$ are  arbitrary matrices over $\mathbb{H}$, $p_{1}$ and $p_{2}$ are the column numbers of $A_{1}$ and $A_{3}$, respectively, $p_{3}$ and $p_{4}$ are the row numbers of $B_{2}$ and $B_{4}$, respectively.

\end{theorem}

\begin{proof}
We separate this system of one-sided coupled Sylvester-type real quaternion matrix equations (\ref{sys04}) into two parts
\begin{align}
  \left\{\begin{array}{c}
A_{1}X_{1}-X_{2}B_{1}=C_{1},\\
A_{2}X_{2}-X_{3}B_{2}=C_{2},
\end{array}
  \right.
\end{align}
and
\begin{align}
  \left\{\begin{array}{c}
A_{3}X_{3}-X_{4}B_{3}=C_{3},\\
A_{4}X_{4}-X_{5}B_{4}=C_{4}.
\end{array}
  \right.
\end{align}
Applying the main idea of Theorem \ref{theorem01}, Lemma \ref{lemma02}, Lemma \ref{lemma03} and Lemma \ref{lemma04}, we can prove Theorem \ref{theorem04}.
\end{proof}

Now we give an example to illustrate Theorem \ref{theorem04}.

\begin{example}
Given the quaternion matrices:
\begin{align*}
A_{1}=\begin{pmatrix}1+\mathbf{k}&-1&2\mathbf{i}+\mathbf{j}\\
0&\mathbf{i}+\mathbf{k}&\mathbf{i}-2\mathbf{j}\\
1+\mathbf{i}&2-\mathbf{i}&1+\mathbf{k}\end{pmatrix}, ~
B_{1}=\begin{pmatrix}-1+\mathbf{k}&\mathbf{i}+\mathbf{k}&\mathbf{j}+\mathbf{k}\\
-2-\mathbf{j}&2\mathbf{i}-\mathbf{j}&-\mathbf{j}+\mathbf{k}\\
1+\mathbf{i}-\mathbf{j}+\mathbf{k}&-1+\mathbf{i}-\mathbf{j}+\mathbf{k}&2\mathbf{k}\end{pmatrix},
\end{align*}
\begin{align*}
A_{2}=\begin{pmatrix}\mathbf{i}&\mathbf{j}&1+2\mathbf{i}+\mathbf{k}\\
\mathbf{k}&\mathbf{i}-\mathbf{j}&-1-2\mathbf{j}+\mathbf{k}\\
\mathbf{i}+\mathbf{k}&\mathbf{i}&2\mathbf{i}-2\mathbf{j}+2\mathbf{k}\end{pmatrix},~
B_{2}=\begin{pmatrix}\mathbf{j}&1+2\mathbf{i}+\mathbf{j}&-\mathbf{i}+\mathbf{k}\\
\mathbf{i}-\mathbf{j}&\mathbf{k}&1+2\mathbf{j}\\
\mathbf{i}&1+2\mathbf{i}+\mathbf{j}+\mathbf{k}&1-\mathbf{i}+2\mathbf{j}+\mathbf{k}\end{pmatrix},
\end{align*}
\begin{align*}
A_{3}=\begin{pmatrix}1+2\mathbf{i}+\mathbf{k}&2-\mathbf{i}-\mathbf{k}&1+\mathbf{j}\\
-1-2\mathbf{i}-\mathbf{j}+\mathbf{k}&-2+\mathbf{i}+\mathbf{j}-\mathbf{k}&-1+\mathbf{j}+\mathbf{k}\\
-\mathbf{k}&\mathbf{k}&-\mathbf{j}\end{pmatrix},
\end{align*}
\begin{align*}
B_{3}=\begin{pmatrix}\mathbf{i}+2\mathbf{j}&1+3\mathbf{j}&\mathbf{j}-3\mathbf{k}\\
-1+\mathbf{i}-2\mathbf{j}&1+\mathbf{i}-3\mathbf{j}&-\mathbf{j}+3\mathbf{k}\\
\mathbf{i}&1&0\end{pmatrix},
\end{align*}
\begin{align*}
A_{4}=\begin{pmatrix}2+3\mathbf{i}+\mathbf{k}&3-\mathbf{j}&\mathbf{i}+\mathbf{j}+\mathbf{k}\\
-3+2\mathbf{i}-\mathbf{j}&3\mathbf{i}-\mathbf{k}&-1-\mathbf{j}+\mathbf{k}\\
-1+5\mathbf{i}-\mathbf{j}+\mathbf{k}&3+3\mathbf{i}-\mathbf{j}-\mathbf{k}&-1+\mathbf{i}+2\mathbf{k}\end{pmatrix},
\end{align*}
\begin{align*}
B_{4}=\begin{pmatrix}1&\mathbf{i}+\mathbf{k}&1+2\mathbf{i}-\mathbf{j}\\
\mathbf{i}&-1-\mathbf{j}&-2+\mathbf{i}-\mathbf{k}\\
1+\mathbf{i}&-1+\mathbf{i}+2\mathbf{k}&-1+3\mathbf{i}-2\mathbf{j}\end{pmatrix},
\end{align*}
\begin{align*}
C_{1}=\begin{pmatrix}-1+4\mathbf{i}-\mathbf{j}-\mathbf{k}&-4+2\mathbf{i}-5\mathbf{j}+6\mathbf{k}&3-2\mathbf{i}+6\mathbf{k}\\
1-5\mathbf{i}-6\mathbf{j}+\mathbf{k}&5+\mathbf{i}-2\mathbf{j}+\mathbf{k}&3-2\mathbf{i}+\mathbf{k}\\
-6-3\mathbf{i}+2\mathbf{j}+3\mathbf{k}&-2-8\mathbf{i}+3\mathbf{j}+11\mathbf{k}&5\mathbf{j}\end{pmatrix},
\end{align*}
\begin{align*}
C_{2}=\begin{pmatrix}2-3\mathbf{i}&8-3\mathbf{j}+4\mathbf{k}&-1+\mathbf{i}-5\mathbf{j}-8\mathbf{k}\\
1-2\mathbf{j}&1-9\mathbf{i}-4\mathbf{j}-2\mathbf{k}&-6+2\mathbf{i}-\mathbf{j}-5\mathbf{k}\\
\mathbf{j}-2\mathbf{k}&6-8\mathbf{i}-5\mathbf{j}+2\mathbf{k}&-7-\mathbf{i}-5\mathbf{j}-4\mathbf{k}\end{pmatrix},
\end{align*}
\begin{align*}
C_{3}=\begin{pmatrix}3+3\mathbf{j}-\mathbf{k}&-3+\mathbf{i}+6\mathbf{j}-2\mathbf{k}&1-4\mathbf{j}-5\mathbf{k}\\
1+\mathbf{i}+\mathbf{j}+2\mathbf{k}&6-4\mathbf{i}+4\mathbf{j}+3\mathbf{k}&-6+13\mathbf{j}-4\mathbf{k}\\
3+4\mathbf{i}-\mathbf{j}+6\mathbf{k}&3-3\mathbf{i}+5\mathbf{j}-4\mathbf{k}&2-\mathbf{i}+5\mathbf{j}-7\mathbf{k}\end{pmatrix},
\end{align*}
\begin{align*}
C_{4}=\begin{pmatrix}-11-5\mathbf{i}-6\mathbf{j}+\mathbf{k}&-1+8\mathbf{i}-2\mathbf{j}+7\mathbf{k}&-10+\mathbf{i}-3\mathbf{j}+6\mathbf{k}\\
5-11\mathbf{i}-3\mathbf{j}-5\mathbf{k}&-6-2\mathbf{i}-5\mathbf{j}-3\mathbf{k}&-2-12\mathbf{i}-3\mathbf{j}+3\mathbf{k}\\
-6-16\mathbf{i}-5\mathbf{j}-4\mathbf{k}&-11+6\mathbf{i}-7\mathbf{j}&-12-7\mathbf{i}-2\mathbf{j}+\mathbf{k}\end{pmatrix}.
\end{align*}
Now we consider the system of one-sided coupled Sylvester-type real quaternion matrix equations (\ref{sys04}). Check that
\begin{align*}
r\begin{pmatrix}C_{i}&A_{i}\\B_{i}&0\end{pmatrix}=r(A_{i})+r(B_{i})=
\left\{\begin{array}{lll}
6, &\mbox{if}~
 i=1\\
4, & \mbox{if}~
 i=2,3\\
 3, & \mbox{if}~
 i=4
\end{array}
 \right.
\end{align*}
\begin{align*}
r\begin{pmatrix}A_{k+1}C_{k}+C_{k+1}B_{k}&A_{k+1}A_{k}\\B_{k+1}B_{k}&0\end{pmatrix}=r(A_{k+1}A_{k})
+r(B_{k+1}B_{k})=
\left\{\begin{array}{lll}
4, &\mbox{if}~
 k=1,2\\
3, & \mbox{if}~
 k=3
\end{array}
 \right.
\end{align*}
\begin{align*}
&r\begin{pmatrix}A_{j+2}A_{j+1}C_{j}+A_{j+2}C_{j+1}B_{j}+C_{j+2}B_{j+1}B_{j}&A_{j+2}A_{j+1}A_{j}\\B_{j+2}B_{j+1}B_{j}&0\end{pmatrix}\nonumber\\&=r(A_{j+2}A_{j+1}A_{j})
+r(B_{j+2}B_{j+1}B_{j})=
\left\{\begin{array}{lll}
4, &\mbox{if}~
 j=1\\
3, & \mbox{if}~
 j=2
\end{array}
 \right.
\end{align*}
\begin{align*}
&r\begin{pmatrix}A_{4}A_{3}A_{2}C_{1}+A_{4}A_{3}C_{2}A_{1}+A_{4}C_{3}A_{2}A_{1}+C_{4}A_{3}A_{2}A_{1}&A_{4}A_{3}A_{2}A_{1}\\B_{4}B_{3}B_{2}B_{1}&0\end{pmatrix}\nonumber\\&
=r(A_{4}A_{3}A_{2}A_{1})
+r(B_{4}B_{3}B_{2}B_{1})=3.
\end{align*}
All the rank equalities in (\ref{equhh061})-(\ref{equhh064}) hold. Hence, the system of one-sided coupled Sylvester-type real quaternion matrix equations (\ref{sys04}) is consistent. Note that
\begin{align*}
X_{1}=\begin{pmatrix}2\mathbf{i}+\mathbf{k}&-1+\mathbf{j}+\mathbf{k}&2+\mathbf{j}\\
-2\mathbf{i}+\mathbf{k}&1+\mathbf{j}+\mathbf{k}&-2+\mathbf{j}\\
2\mathbf{k}&2\mathbf{j}+2\mathbf{k}&2\mathbf{j}\end{pmatrix}
~
X_{2}=\begin{pmatrix}1&-1+\mathbf{j}&\mathbf{i}+\mathbf{k}\\
2&-2-\mathbf{j}&2\mathbf{i}-\mathbf{k}\\
-1&1+2\mathbf{j}&-\mathbf{i}+2\mathbf{k}\end{pmatrix},
\end{align*}
\begin{align*}
X_{3}=\begin{pmatrix}\mathbf{i}+\mathbf{j}&1+2\mathbf{i}+\mathbf{k}&2\mathbf{k}\\
1&\mathbf{k}&1\\
\mathbf{i}&0&1+\mathbf{k}\end{pmatrix}
~
X_{4}=\begin{pmatrix}-1+\mathbf{i}+\mathbf{k}&1+\mathbf{k}&\mathbf{i}+\mathbf{k}\\
-1-\mathbf{i}+\mathbf{k}&\mathbf{i}+\mathbf{k}&-1+\mathbf{k}\\
-2+2\mathbf{k}&1+\mathbf{i}+2\mathbf{k}&-1+\mathbf{i}+2\mathbf{k}\end{pmatrix},
\end{align*}and
\begin{align*}
X_{5}=\begin{pmatrix}1&-1+\mathbf{j}&\mathbf{i}+\mathbf{k}\\
2&-2+2\mathbf{j}&2\mathbf{i}+2\mathbf{k}\\
3&-3-\mathbf{j}&3\mathbf{i}-\mathbf{k}\end{pmatrix}
\end{align*}
satisfy the system (\ref{sys04}).

\end{example}

\section{\textbf{Some solvability conditions and the general solution to
system (\ref{sys05})}}

In this section, we consider the solvability conditions and the general solution to the
system of one-sided coupled Sylvester-type real quaternion matrix equations (\ref{sys05}). For simplicity, put
\begin{align*}
A_{11}=R_{B_{2}}B_{1},~B_{11}=R_{A_{1}}A_{2},~C_{11}=B_{1}L_{A_{11}},~D_{11}=R_{A_{1}}(R_{A_{2}}C_{2}B_{2}^{\dag}B_{1}-C_{1})L_{A_{11}},
\end{align*}
\begin{align*}
A_{22}=R_{(A_{4}L_{A_{3}})}A_{4},B_{22}=B_{3}L_{B_{4}},C_{22}=R_{(A_{4}L_{A_{3}})}(C_{4}-A_{4}A_{3}^{\dag}C_{3})L_{B_{4}},
\end{align*}
\begin{align*}
A_{33}=(L_{A_{2}},~-A_{3}L_{A_{22}}),\qquad B_{33}=\begin{pmatrix}R_{C_{11}}B_{2}\\-R_{B_{3}}\end{pmatrix},
\end{align*}
\begin{align*}
E_{1}=-R_{A_{3}}C_{3}B_{3}^{\dag}+A_{3}A_{22}^{\dag}C_{22}B_{22}^{\dag}-A_{2}^{\dag}C_{2}-B_{11}^{\dag}D_{11}C_{11}^{\dag}B_{2},
\end{align*}
\begin{align*}
A=R_{A_{33}}L_{B_{11}},~B=B_{2}L_{B_{33}},~C=-R_{A_{33}}A_{3},~D=R_{B_{22}}L_{B_{33}},
\end{align*}
\begin{align*}
E=R_{A_{33}}E_{1}L_{B_{33}},~M=R_{A}C,~N=DL_{B},~S=CL_{M}.
\end{align*}

\begin{theorem}\label{theorem71}
Let $A_{i},B_{i},$ and $C_{i}(i=1,2,3,4)$ be given. Then the following statements are equivalent:\\$(1)$ The system of one-sided coupled Sylvester-type real quaternion matrix equations (\ref{sys05}) is consistent.\\$(2)$
\begin{align}\label{equhh071}
r\begin{pmatrix}C_{i}&A_{i}\\B_{i}&0\end{pmatrix}=r(A_{i})+r(B_{i}),\quad (i=1,2,3,4),
\end{align}
\begin{align}
r\begin{pmatrix}C_{1}&C_{2}&A_{1}&A_{2}\\B_{1}&B_{2}&0&0\end{pmatrix}=r(A_{1},~A_{2})+r(B_{1},~B_{2}),
\end{align}
\begin{align}
r\begin{pmatrix}C_{3}&A_{3}\\C_{4}&A_{4}\\B_{3}&0\\B_{4}&0\end{pmatrix}=r\begin{pmatrix}A_{3}\\A_{4}\end{pmatrix}+r\begin{pmatrix}B_{3}\\B_{4}\end{pmatrix},
\end{align}
\begin{align}
r\begin{pmatrix}C_{1}&A_{2}C_{3}+C_{2}B_{3}&A_{1}&A_{2}A_{3}\\B_{1}&B_{2}B_{3}&0&0\end{pmatrix}=
r(A_{1},~A_{2}A_{3})+r(B_{1},~B_{2}B_{3}),
\end{align}
\begin{align}
r\begin{pmatrix}A_{2}C_{3}+C_{2}B_{3}&A_{2}A_{3}\\C_{4}&A_{4}\\B_{2}B_{3}&0\\B_{4}&0\end{pmatrix}=
r\begin{pmatrix}A_{2}A_{3}\\A_{4}\end{pmatrix}+r\begin{pmatrix}B_{2}B_{3}\\B_{4}\end{pmatrix},
\end{align}
\begin{align}
r\begin{pmatrix}A_{2}C_{3}+C_{2}B_{3}&A_{2}A_{3}\\B_{2}B_{3}&0\end{pmatrix}=
r(A_{2}A_{3})+r(B_{2}B_{3}),
\end{align}
\begin{align}\label{equhh077}
r\begin{pmatrix}C_{1}&A_{2}C_{3}+C_{2}B_{3}&A_{1}&A_{2}A_{3}\\0&C_{4}&0&A_{4}\\B_{1}&B_{2}B_{3}&0&0\\0&B_{4}&0&0\end{pmatrix}=
r\begin{pmatrix}A_{1}&A_{2}A_{3}\\0&A_{4}\end{pmatrix}+r\begin{pmatrix}B_{1}&B_{2}B_{3}\\0&B_{4}\end{pmatrix}.
\end{align}

$(3)$
\begin{align*}
R_{A_{2}}C_{2}L_{B_{2}}=0,~D_{11}L_{C_{11}}=0,~R_{B_{11}}D_{11}=0,
\end{align*}
\begin{align*}
R_{A_{3}}C_{3}L_{B_{3}}=0,~R_{A_{22}}C_{22}=0,~C_{22}L_{B_{22}}=0,
\end{align*}
\begin{align*}
R_{M}R_{A}E=0,~EL_{B}L_{N}=0,~R_{A}EL_{D}=0,~R_{C}EL_{B}=0.
\end{align*}

In this case, the general solution to the system of one-sided coupled Sylvester-type real quaternion matrix equations (\ref{sys05}) can be expressed as
\begin{align*}
X_{1}=A_{1}^{\dag}(C_{1}-R_{A_{2}}C_{2}B_{2}^{\dag}B_{1}+A_{2}U_{1}B_{1})+W_{4}A_{11}+L_{A_{1}}W_{6},
\end{align*}
\begin{align*}
X_{2}=-R_{A_{2}}C_{2}B_{2}^{\dag}+A_{2}U_{1}+V_{1}R_{B_{2}},
~
X_{4}=A_{3}^{\dag}C_{3}+V_{2}B_{3}+L_{A_{3}}U_{2},
\end{align*}
\begin{align*}
X_{5}=-R_{(A_{4}L_{A_{3}})}(C_{4}-A_{4}A_{3}^{\dag}C_{3}-A_{4}V_{2}B_{3})B_{4}^{\dag}+A_{4}L_{A_{3}}T_{1}+T_{3}R_{B_{4}},
\end{align*}
\begin{align*}
X_{3}=A_{2}^{\dag}C_{2}+U_{1}B_{2}+L_{A_{2}}W_{1},
~\mbox{or}~
X_{3}=-R_{A_{3}}C_{3}B_{3}^{\dag}+A_{3}V_{2}+T_{6}R_{B_{3}},
\end{align*}
where
\begin{align*}
U_{1}=B_{11}^{\dag}D_{11}C_{11}^{\dag}+L_{B_{11}}W_{2}+W_{3}R_{C_{11}},
\end{align*}
\begin{align*}
V_{1}=-R_{A_{1}}(C_{1}-R_{A_{2}}C_{2}B_{2}^{\dag}B_{1}+A_{2}U_{1}B_{1})A_{11}^{\dag}+A_{1}W_{4}+W_{5}R_{A_{11}},
\end{align*}
\begin{align*}
V_{2}=A_{22}^{\dag}C_{22}B_{22}^{\dag}+L_{A_{22}}T_{4}+T_{5}R_{B_{22}},
\end{align*}
\begin{align*}
U_{2}=(A_{4}L_{A_{3}})^{\dag}(C_{4}-A_{4}A_{3}^{\dag}C_{3}-A_{4}V_{2}B_{3})+T_{1}B_{4}+L_{(A_{4}L_{A_{3}})}T_{2},
\end{align*}
\begin{align*}
W_{1}=(I_{p_{1}},~0)[A_{33}^{\dag}(E_{1}-L_{B_{11}}W_{2}B_{2}+A_{3}T_{5}R_{B_{22}})-A_{33}^{\dag}Z_{7}B_{33}+L_{A_{33}}Z_{6}],
\end{align*}
\begin{align*}
T_{4}=(0,~I_{p_{2}})[A_{33}^{\dag}(E_{1}-L_{B_{11}}W_{2}B_{2}+A_{3}T_{5}R_{B_{22}})-A_{33}^{\dag}Z_{7}B_{33}+L_{A_{33}}Z_{6}],
\end{align*}
\begin{align*}
W_{3}=[R_{A_{33}}(E_{1}-L_{B_{11}}W_{2}B_{2}+A_{3}T_{5}R_{B_{22}})B_{33}^{\dag}
+A_{33}A_{33}^{\dag}Z_{7}+Z_{8}R_{B_{33}}]\begin{pmatrix}I_{p_{3}}\\0\end{pmatrix},
\end{align*}
\begin{align*}
T_{6}=[R_{A_{33}}(E_{1}-L_{B_{11}}W_{2}B_{2}+A_{3}T_{5}R_{B_{22}})B_{33}^{\dag}
+A_{33}A_{33}^{\dag}Z_{7}+Z_{8}R_{B_{33}}]\begin{pmatrix}0\\I_{p_{4}}\end{pmatrix},
\end{align*}
\begin{align*}
W_{2}=
A^{\dag}EB^{\dag}-A^{\dag}CM^{\dag}EB^{\dag}-A^{\dag}SC^{\dag
}EN^{\dag}DB^{\dag}-A^{\dag}SZ_{1}R_{N}DB^{\dag}+L_{A}Z_{2}+Z_{3}%
R_{B},
\end{align*}
\begin{align*}
T_{5}=M^{\dag}ED^{\dag}+S^{\dag}SC^{\dag}EN^{\dag}+L_{M}L_{S}Z_{4}%
+L_{M}Z_{1}R_{N}+Z_{5}R_{D},
\end{align*}
the remaining $W_{j},T_{j},Z_{j}$ are  arbitrary matrices over $\mathbb{H}$, $p_{1}$ and $p_{2}$ are the column numbers of $A_{2}$ and $A_{4}$, respectively, $p_{3}$ and $p_{4}$ are the row numbers of $B_{1}$ and $B_{3}$, respectively.

\end{theorem}

\begin{proof}
We separate this system of one-sided coupled Sylvester-type real quaternion matrix equations (\ref{sys05}) into two parts
\begin{align}
  \left\{\begin{array}{c}
A_{1}X_{1}-X_{2}B_{1}=C_{1},\\
A_{2}X_{3}-X_{2}B_{2}=C_{2},
\end{array}
  \right.
\end{align}
and
\begin{align}
  \left\{\begin{array}{c}
A_{3}X_{4}-X_{3}B_{3}=C_{3},\\
A_{4}X_{4}-X_{5}B_{4}=C_{4}.
\end{array}
  \right.
\end{align}
Applying the main idea of Theorem \ref{theorem01}, Lemma \ref{lemma01}, Lemma \ref{lemma05}, Lemma \ref{lemma03} and Lemma \ref{lemma04}, we can prove Theorem \ref{theorem71}.
\end{proof}

Now we give an example to illustrate Theorem \ref{theorem71}.

\begin{example}
Given the quaternion matrices:
\begin{align*}
A_{1}=\begin{pmatrix}\mathbf{i}+\mathbf{j}&-\mathbf{j}&\mathbf{i}+\mathbf{k}\\
\mathbf{k}&1+\mathbf{k}&0\\
1&0&1+\mathbf{j}\end{pmatrix}, ~
B_{1}=\begin{pmatrix}1+\mathbf{j}+\mathbf{k}&-1-\mathbf{k}&\mathbf{i}+\mathbf{j}\\
2\mathbf{k}&1&1+\mathbf{i}+\mathbf{j}\\
2&2+\mathbf{i}+\mathbf{j}&\mathbf{k}\end{pmatrix},
\end{align*}
\begin{align*}
A_{2}=\begin{pmatrix}1&1+\mathbf{i}+\mathbf{j}&2+2\mathbf{i}+\mathbf{k}\\
1-2\mathbf{i}+\mathbf{k}&\mathbf{j}&1\\
\mathbf{i}&1&\mathbf{j}\end{pmatrix},~
B_{2}=\begin{pmatrix}1&-1+\mathbf{j}&\mathbf{i}+\mathbf{k}\\
\mathbf{i}&-\mathbf{i}-\mathbf{j}&-1-\mathbf{k}\\
1+\mathbf{i}&-1-\mathbf{i}&-1+\mathbf{i}\end{pmatrix},
\end{align*}
\begin{align*}
A_{3}=\begin{pmatrix}\mathbf{j}&\mathbf{i}-\mathbf{j}&1+\mathbf{k}\\
1+\mathbf{k}&0&\mathbf{i}+\mathbf{j}\\
1+\mathbf{j}+\mathbf{k}&\mathbf{i}-\mathbf{j}&1+\mathbf{i}+\mathbf{j}+\mathbf{k}\end{pmatrix},
~
B_{3}=\begin{pmatrix}\mathbf{j}+2\mathbf{k}&1+\mathbf{j}-\mathbf{k}&\mathbf{i}+\mathbf{j}\\
-\mathbf{j}-2\mathbf{k}&-1-\mathbf{j}+\mathbf{k}&-\mathbf{i}-\mathbf{j}\\
2\mathbf{j}+4\mathbf{k}&2\mathbf{j}-2\mathbf{k}&2\mathbf{j}\end{pmatrix},
\end{align*}
\begin{align*}
A_{4}=\begin{pmatrix}-\mathbf{k}&\mathbf{i}+\mathbf{j}+\mathbf{k}&2\mathbf{i}-2\mathbf{k}\\
1+\mathbf{k}&1-\mathbf{j}-\mathbf{k}&1+2\mathbf{k}\\
1&1+\mathbf{i}&1+2\mathbf{i}\end{pmatrix},
~
B_{4}=\begin{pmatrix}1-\mathbf{j}&\mathbf{i}-\mathbf{k}&-\mathbf{i}-\mathbf{k}\\
\mathbf{i}+\mathbf{j}&-1+\mathbf{k}&1+\mathbf{k}\\
1+\mathbf{i}+2\mathbf{j}&-1+\mathbf{i}+2\mathbf{k}&1-\mathbf{i}+2\mathbf{k}\end{pmatrix},
\end{align*}
\begin{align*}
C_{1}=\begin{pmatrix}-3-\mathbf{j}-4\mathbf{k}&-1+2\mathbf{i}+3\mathbf{k}&1-\mathbf{j}\\
2-\mathbf{i}+\mathbf{j}-3\mathbf{k}&-1+\mathbf{i}-\mathbf{k}&2-\mathbf{i}+\mathbf{k}\\
4+\mathbf{i}-3\mathbf{k}&1-2\mathbf{i}+\mathbf{j}&-2-2\mathbf{i}-\mathbf{j}-\mathbf{k}\end{pmatrix},
\end{align*}
\begin{align*}
C_{2}=\begin{pmatrix}-2-2\mathbf{i}-8\mathbf{j}+5\mathbf{k}&11\mathbf{j}+7\mathbf{k}&2+\mathbf{i}-8\mathbf{j}+5\mathbf{k}\\
-1-2\mathbf{i}-2\mathbf{j}+4\mathbf{k}&-1+2\mathbf{i}+9\mathbf{j}-\mathbf{k}&4+\mathbf{i}-2\mathbf{j}+7\mathbf{k}\\
1-2\mathbf{j}+\mathbf{k}&-3+\mathbf{i}+5\mathbf{j}-\mathbf{k}&2+2\mathbf{i}+\mathbf{j}+5\mathbf{k}\end{pmatrix},
\end{align*}
\begin{align*}
C_{3}=\begin{pmatrix}-4+9\mathbf{i}+4\mathbf{j}-4\mathbf{k}&-8-3\mathbf{i}+\mathbf{k}&-5+\mathbf{j}-6\mathbf{k}\\
8+5\mathbf{i}-\mathbf{j}+6\mathbf{k}&-2+4\mathbf{i}+6\mathbf{j}-2\mathbf{k}&-2+6\mathbf{i}+2\mathbf{k}\\
4+14\mathbf{i}+3\mathbf{j}+2\mathbf{k}&-10+\mathbf{i}+6\mathbf{j}+\mathbf{k}&-7+6\mathbf{i}+\mathbf{j}-4\mathbf{k}\end{pmatrix},
\end{align*}
\begin{align*}
C_{4}=\begin{pmatrix}-3+2\mathbf{i}-3\mathbf{j}-6\mathbf{k}&2+4\mathbf{i}+\mathbf{j}-3\mathbf{k}&-4\mathbf{i}-\mathbf{j}-\mathbf{k}\\
-3-4\mathbf{i}-7\mathbf{j}+5\mathbf{k}&7-5\mathbf{i}-2\mathbf{j}-2\mathbf{k}&-10+8\mathbf{i}+3\mathbf{j}-9\mathbf{k}\\
-4+\mathbf{i}-3\mathbf{j}+3\mathbf{k}&4+3\mathbf{i}-3\mathbf{j}&-7+2\mathbf{i}-2\mathbf{j}-3\mathbf{k}\end{pmatrix}.
\end{align*}
Now we consider the system of one-sided coupled Sylvester-type real quaternion matrix equations (\ref{sys05}). Check that
\begin{align*}
r\begin{pmatrix}C_{i}&A_{i}\\B_{i}&0\end{pmatrix}=r(A_{i})+r(B_{i})=
\left\{\begin{array}{lll}
6, &\mbox{if}~
 i=1\\
5, & \mbox{if}~
 i=2\\
 4, & \mbox{if}~
 i=3,4
\end{array}
 \right.
\end{align*}
\begin{align*}
r\begin{pmatrix}C_{1}&C_{2}&A_{1}&A_{2}\\B_{1}&B_{2}&0&0\end{pmatrix}=r(A_{1},~A_{2})+r(B_{1},~B_{2})=6,
\end{align*}
\begin{align*}
r\begin{pmatrix}C_{3}&A_{3}\\C_{4}&A_{4}\\B_{3}&0\\B_{4}&0\end{pmatrix}
=r\begin{pmatrix}A_{3}\\A_{4}\end{pmatrix}+r\begin{pmatrix}B_{3}\\B_{4}\end{pmatrix}=6,
\end{align*}
\begin{align*}
r\begin{pmatrix}C_{1}&A_{2}C_{3}+C_{2}B_{3}&A_{1}&A_{2}A_{3}\\B_{1}&B_{2}B_{3}&0&0\end{pmatrix}=
r(A_{1},~A_{2}A_{3})+r(B_{1},~B_{2}B_{3})=6,
\end{align*}
\begin{align*}
r\begin{pmatrix}A_{2}C_{3}+C_{2}B_{3}&A_{2}A_{3}\\C_{4}&A_{4}\\B_{2}B_{3}&0\\B_{4}&0\end{pmatrix}=
r\begin{pmatrix}A_{2}A_{3}\\A_{4}\end{pmatrix}+r\begin{pmatrix}B_{2}B_{3}\\B_{4}\end{pmatrix}=6,
\end{align*}
\begin{align*}
r\begin{pmatrix}A_{2}C_{3}+C_{2}B_{3}&A_{2}A_{3}\\B_{2}B_{3}&0\end{pmatrix}=
r(A_{2}A_{3})+r(B_{2}B_{3})=10,
\end{align*}
\begin{align*}
r\begin{pmatrix}C_{1}&A_{2}C_{3}+C_{2}B_{3}&A_{1}&A_{2}A_{3}\\0&C_{4}&0&A_{4}\\B_{1}&B_{2}B_{3}&0&0\\0&B_{4}&0&0\end{pmatrix}=
r\begin{pmatrix}A_{1}&A_{2}A_{3}\\0&A_{4}\end{pmatrix}+r\begin{pmatrix}B_{1}&B_{2}B_{3}\\0&B_{4}\end{pmatrix}=4.
\end{align*}
All the rank equalities in (\ref{equhh071})-(\ref{equhh077}) hold. Hence, the system of one-sided coupled Sylvester-type real quaternion matrix equations (\ref{sys05}) is consistent. Note that
\begin{align*}
X_{1}=\begin{pmatrix}1+\mathbf{i}+\mathbf{j}&1-2\mathbf{i}+2\mathbf{j}&\mathbf{k}\\
1&\mathbf{i}+\mathbf{j}&2\\
\mathbf{i}+\mathbf{k}&1+2\mathbf{k}&-\mathbf{k}\end{pmatrix}
~
X_{2}=\begin{pmatrix}\mathbf{j}&1-\mathbf{j}&\mathbf{i}+\mathbf{k}\\
\mathbf{i}+\mathbf{j}&1-\mathbf{j}&\mathbf{k}\\
-1&2+\mathbf{k}&\mathbf{i}+\mathbf{j}\end{pmatrix},
\end{align*}
\begin{align*}
X_{3}=\begin{pmatrix}-1-\mathbf{j}+\mathbf{k}&-1+\mathbf{j}+\mathbf{k}&\mathbf{i}-\mathbf{j}\\
-1+\mathbf{j}+\mathbf{k}&1+\mathbf{j}-\mathbf{k}&-\mathbf{i}+\mathbf{k}\\
2\mathbf{k}&2\mathbf{j}&-\mathbf{j}+\mathbf{k}\end{pmatrix}
~
X_{4}=\begin{pmatrix}1&1+\mathbf{j}&\mathbf{i}+\mathbf{k}\\
\mathbf{i}&\mathbf{i}-\mathbf{j}&-1-\mathbf{k}\\
1+\mathbf{i}&1+\mathbf{i}&-1+\mathbf{i}\end{pmatrix},
\end{align*}and
\begin{align*}
X_{5}=\begin{pmatrix}1+\mathbf{j}&1+\mathbf{j}&\mathbf{i}+\mathbf{k}\\
1+\mathbf{i}+\mathbf{k}&2-\mathbf{i}+\mathbf{k}&3\\
1+3\mathbf{i}+\mathbf{j}&\mathbf{k}&1\end{pmatrix}
\end{align*}
satisfy the system (\ref{sys05}).

\end{example}

\section{Conclusion}

We have provided some necessary and sufficient conditions for the existence and the
general solutions to the systems of four coupled one sided Sylvester-type real quaternion matrix equations (\ref{sys01})-(\ref{sys05}), respectively. Moreover, we have presented some numerical examples. It is worthy to say that the main results of
this paper can be generalized to an arbitrary division ring with an
involutive antiautomorphism.


\begin{thebibliography}{99}




\bibitem {jkb}J.K. Baksalary, R. Kala, The matrix equation $AX-YB=C$, \textit{Linear Algebra Appl}. 25 (1979)\textbf{ }41-43.


\bibitem {JKB2}J.K. Baksalary, R. Kala, The matrix equation $AXB+CYD=E$, \textit{Linear Algebra Appl}. 30 (1980)\textbf{ }141-147.


\bibitem {N. LE Bihan}N.L. Bihan, J. Mars, Singular value decomposition of
quaternion matrices: A new tool for vector-sensor signal processing,
\textit{Signal Processing}, 84 (7) (2004)\textbf{ }1177-1199.

\bibitem{mao8} J. Chen , R. Patton, H. Zhang, Design unknown input observers and robust fault detection filters, \textit{Int. J. of control}. 63 (1996) \textbf{ }85--105.


\bibitem {xibanya01}F. De Ter$\acute{a}$n, F.M. Dopico, N. Guillery, D. Montealegre, N. Reyes, The solution of the equation $AX+ X^{*}B= 0$, \textit{Linear Algebra Appl}. 438 (7) (2013)\textbf{ }  2817--2860.

\bibitem {xibanyalama}F. De Ter$\acute{a}$n, The solution of the equation $AX+B X^{*}=0$, \textit{Linear and Multilinear Algebra} 61 (12) (2013)\textbf{ }  1605--1628.


\bibitem {Dmytryshyn}A. Dmytryshyn, B. K\aa gstr\"{o}m, Coupled Sylvester-type matrix equations and block diagonalization, \textit{SIAM J. Matrix Anal. Appl.} 36 (2)(2015)\textbf{ }
580--593.


 \bibitem {helaa}   Z.H. He, O.M. Agudelo, Q.W. Wang, B. De Moor,  Two-sided coupled generalized Sylvester matrix
equations solving using a simultaneous decomposition for fifteen matrices, \textit{Linear Algebra Appl.} 496 (2016)\textbf{ }
549-593.

\bibitem {wanghe2}Z.H. He, Q.W. Wang, A real quaternion matrix equation with with applications,
\textit{Linear and Multilinear Algebra} 61 (2013)\textbf{ }725--740.


\bibitem {hewangamc2017}Z.H. He, Q.W. Wang, Y. Zhang, Simultaneous decomposition of quaternion matrices involving $\eta$-Hermicity with applications,
\textit{Appl. Math. Comput}.  298 (2017)\textbf{ }13--35.

\bibitem {wanghe3}Z.H. He, Q.W. Wang, The $\eta$-bihermitian solution to a system of real
quaternion matrix equations, \textit{Linear and Multilinear Algebra} 62 (2014)\textbf{ }1509--1528.

\bibitem{hewang4} Z.H. He, Q.W. Wang, The general solutions to some systems
of matrix equations, \textit{Linear and Multilinear Algebra}
63 (10) (2015) \textbf{ }2017--2032.


\bibitem {heac2017}Z.H. He, Q.W. Wang, A system of periodic discrete-time coupled Sylvester
quaternion matrix equations, \textit{ Algebra Colloquium} \textbf{24} (2017) 169--180.


\bibitem {shangdaxuebao}Z.H. He, Q.W. Wang, A pair of mixed generalized Sylvester matrix equations,
\textit{Journal of Shanghai University (Natural Science)}, 20 (2014)\textbf{ } 138-156.

\bibitem {Jonsson01}I. Jonsson, B. K\aa gstr\"{o}m, Recursive blocked algorithms for solving triangular systems-Part I: One-sided and coupled Sylvester-type matrix equations,
\textit{ACM Transactions on Mathematical Software}. 284 (2002)\textbf{ }  392-415.


\bibitem {Jonsson02}I. Jonsson, B. K\aa gstr\"{o}m, Recursive blocked algorithms for solving triangular systems-Part II: Two-sided and generalized Sylvester and Lyapunov matrix equations, \textit{ACM Transactions on Mathematical Software}. 28 (2002)\textbf{ }  416-435.

\bibitem{mao16} B. K\aa gstr\"{o}m, L. Westin, Generalized Schur methods with condition estimators for solving the generalized Sylvester equation, \textit{IEEE Trans. on Automatic Control}.  34 (7) (1989) 745--751.

\bibitem {mao1} O. Kamen\'ik, Solving SDGE Models: A New Algorithm for the Sylvester Equation, \textit{Comput. Econom}. 25 (2005)\textbf{ }167--187.


\bibitem {SangGuLee}S.G. Lee, Q.P. Vu, Simultaneous solutions of matrix equations and simultaneous
equivalence of matrices, \textit{Linear Algebra Appl}. 437 (2012)\textbf{ } 2325-2339.

\bibitem {SangGuLee2}S.G. Lee, Q.P. Vu, Simultaneous solutions of Sylvester equations and
idempotent matrices separating the joint spectrum, \textit{Linear Algebra Appl}. 435 (2011)\textbf{ } 2097--2209.

\bibitem {S. De Leo}S.D. Leo, G. Scolarici, Right eigenvalue equation in
quaternionic quantum mechanics, \textit{J. Phys}. A 33 (2000)\textbf{ }2971-2995.


\bibitem {GPH}G. Marsaglia, G.P.H. Styan, Equalities and inequalities for
ranks of matrices, \textit{Linear Multilinear Algebra}. 2
(1974)\textbf{ }269--292.


\bibitem {Roth}W.E. Roth, The equation $AX-YB=C$ and $AX-XB=C$ in matrices,
\textit{Proc. Amer. Math. Soc.} 3 (1952)\textbf{ }392-396.

\bibitem {mao5} A. Shahzad, B.L. Jones, E.C. Kerrigan, G.A. Constantinides,
An efficient algorithm for the solution of a coupled Sylvester equation appearing
in descriptor systems, \textit{Automatica}. 47 (2011)\textbf{ }244--248.

\bibitem {VLS}V.L. Syrmos, F.L. Lewis,  Output feedback eigenstructure assignment
using two Sylvester equations,
\textit{IEEE Trans. on Automatic Control}. 38(1993)\textbf{ } 495-499.



\bibitem{mao10} V.L. Syrmos, F.L. Lewis, Coupled and constrained Sylvester equations in System design, \textit{Circuits Systems Signal Process}. 13 (6) (1994) \textbf{ }663-694.


\bibitem {Took1}C.C. Took, D.P. Mandic, Augmented second-order
statistics of quaternion random signals, \textit{Signal Processing} 91
(2011)\textbf{ }214-224.

\bibitem {Took2}C.C. Took, D.P. Mandic, The quaternion LMS algorithm
for adaptive filtering of hypercomplex real world processes, \textit{IEEE
Trans. Signal Process}. 57 (2009)\textbf{ }1316-1327.

\bibitem {Took3}C.C. Took, D.P. Mandic, Quaternion-valued stochastic
gradient-based adaptive IIR filtering, \textit{IEEE Trans. Signal Process}. 58
(7) (2010)\textbf{ }3895-3901.

\bibitem {Took4}C.C. Took, D.P. Mandic, F.Z. Zhang, On the unitary
diagonalization of a special class of quaternion matrices, \textit{Appl. Math.
Lett}. 24 (2011)\textbf{ }1806-1809.

\bibitem {Varga} A. Varga, Robust pole assignment via Sylvester equation based state feedback parametrization.
\textit{ Computer-Aided Control System Design, 2000. CACSD 2000. IEEE International Symposium on}. 57(2000)\textbf{ }13-18.

\bibitem {J.W}J.W. van der Woude, Almost noninteracting control by measurement
feedback, \textit{Systems Control Lett}. 9 (1987)\textbf{ }7--16.

\bibitem {wanghe4444444}Q.W. Wang, Z.H. He, Some matrix equations with applications,
\textit{Linear Multilinear Algebra.} 60 (2012)\textbf{ }1327--1353.


\bibitem {wangheauto}Q.W. Wang, Z.H. He, Solvability conditions and general solution for the mixed
Sylvester equations,
\textit{Automatica}. 49 (2013)\textbf{ }2713--2719.


\bibitem {auto001}Q.W. Wang, Z.H. He, Systems of coupled generalized Sylvester matrix equations,
\textit{Automatica}.50 (2014)\textbf{ }2840--2844.



\bibitem {Q.W5}Q.W. Wang, A system of matrix equations and a linear matrix equation over
arbitrary regular rings with identity, \textit{Linear Algebra Appl.}
384 (2004)\textbf{ }43--54.


\bibitem {Q.W12}Q.W. Wang, J.H. Sun, S.Z. Li, Consistency for
bi(skew)symmetric solutions to systems of generalized Sylvester equations over
a finite central algebra, \textit{Linear Algebra Appl.} 353 (2002)\textbf{ }169-182.


\bibitem {QWWangandyushaowen}Q.W. Wang, J.W. van der Woude, S.W. Yu, An
equivalence canonical form of a matrix triplet over an arbitrary division ring
with applications, \textit{Sci. China Math.} 54 (5)(2011)\textbf{ }907--924.




\bibitem {zhangxia}Q.W. Wang, X. Zhang, J.W. van der Woude,
A new simultaneous decomposition of a matrix
quaternity over an arbitrary division ring with
applications, \textit{Comm. Algebra}. 40 (2012)\textbf{ } 2309--2342.


\bibitem {41} Q.W. Wang, The general solution to a system of real
quaternion matrix equations, \textit{ Comput. Math. Appl.} 49
(2005) \textbf{ }665--675.


\bibitem {42} Q.W. Wang, H.X. Chang, C.Y. Lin, P-(skew)symmetric common
solutions to a pair of quaternion matrix equations, \textit{ Appl.
Math. Comput.}  195 (2008)\textbf{ } 721--732.

\bibitem {Q.W11}Q.W. Wang, Bisymmetric and centrosymmetric solutions to system
of real quaternion matrix equations, \textit{Comput. Math. Appl}. 49
(2005)\textbf{ }641--650.

\bibitem {Q.W5}Q.W. Wang, J.W. van der Woude, H.X. Chang, A system of real
quaternion matrix equations with applications, \textit{Linear Algebra Appl.}
431 (2009)\textbf{ }2291--2303.



\bibitem {H.K.Wimmer}H.K. Wimmer, Consistency of a pair of generalized Sylvester equations.
\textit{IEEE Trans. on Automatic Control}. 39(1994)\textbf{ }1014-1016.


\bibitem {yuan4}S.F. Yuan, Q.W. Wang, Two special kinds of least squares
solutions for the quaternion matrix equation $AXB+CXD=E$ , \textit{Electron.
J. Linear Algebra}. 23 (2012)\textbf{ }257--274.


\bibitem {yuan1}S.F. Yuan, Q.W. Wang, L-structured quaternion matrices and quaternion linear matrix equations, \textit{Linear and Multilinear Algebra}
64 (2016)\textbf{ }321--339.


\bibitem  {zhangyong}Y.N. Zhang, D.C. Jiang, J. Wang,
A recurrent neural network for solving Sylvester
equation with time-varying coefficients.
\textit{IEEE Trans. Neural Networks}. 13 (5)(2002)\textbf{ }1053-1063.



 \bibitem {F. Zhang}F.Z. Zhang, Quaternions and matrices of quaternions,
\textit{Linear Algebra Appl.} 251 (1997)\textbf{ }21-57.

\bibitem {wangronghao}Y. Zhang, R.H. Wang, The exact solution of a system of quaternion matrix equations
involving $\eta$-Hermicity, \textit{Appl. Math. Comput}. 222
(2013)\textbf{ }201--209.


\end{thebibliography}
\end{document}